\documentclass[journal,twoside,web]{ieeecolor}
\usepackage{color}
\usepackage{graphicx}
\usepackage{generic}
\usepackage{alltt}
\usepackage{cite}
\usepackage{booktabs}
\usepackage{makecell}
\usepackage{longtable}
\usepackage{multirow}
\usepackage{mathrsfs}
\usepackage{amsmath,amsfonts,amssymb,epstopdf}
\usepackage[linesnumbered,boxed,ruled,commentsnumbered]{algorithm2e}
\usepackage{textcomp}
\usepackage[normalem]{ulem}
\def\BibTeX{{\rm B\kern-.05em{\sc i\kern-.025em b}\kern-.08em
    T\kern-.1667em\lower.7ex\hbox{E}\kern-.125emX}}
\begin{document}
\newtheorem{theorem}{Theorem}
\newtheorem{lemma}{Lemma}
\newtheorem{definition}{Definition}
\newtheorem{remark}{Remark}
\newtheorem{corollary}{Corollary}
\newtheorem{proposition}{Proposition}

\title{A Class of Convex Optimization-Based Recursive Algorithms for Identification of Stochastic Systems}
\author{Mingxia Ding, Wenxiao Zhao, and Tianshi Chen
\thanks{This work was supported by the National Key Research and Development Program of
China (2018YFA0703800), by the Chinese Academy of Sciences (CAS) Project for Young Scientists
in Basic Research under grant YSBR-008, and by the Strategic Priority Research Program of Chinese
Academy of Sciences under grant XDA27000000.}
\thanks{Mingxia Ding and Wenxiao Zhao are with the Key Laboratory of Systems and Control,
Academy of Mathematics and Systems Science, Chinese Academy of
Sciences, Beijing 100190, China, and also with the School of Mathematical Sciences, University of Chinese Academy of Sciences, Beijing
100049, China (e-mail: dingmingxia@amss.ac.cn; wxzhao@amss.ac.cn). }
\thanks{Tianshi Chen is with the School of Science and Engineering and Shenzhen Research Institute of Big Data, The Chinese University of Hong
Kong, Shenzhen, 518172, Shenzhen, China (e-mail: tschen@cuhk.edu.cn).}}
\maketitle
\begin{abstract}
Focusing on identification, this paper develops a class of convex optimization-based criteria and correspondingly the recursive algorithms to estimate the parameter vector $\theta^{*}$ of a stochastic dynamic system. Not only do the criteria include the classical least-squares estimator but also the $L_l=|\cdot|^l, l\geq 1$, the Huber, the Log-cosh, and the Quantile costs as special cases. First, we prove that the minimizers of the convex optimization-based criteria converge to $\theta^{*}$ with probability one. Second, the recursive algorithms are proposed to find the estimates, which minimize the convex optimization-based criteria, and it is shown that these estimates also converge to the true parameter vector with probability one. Numerical examples are given, justifying the performance of the proposed algorithms including the strong consistency of the estimates, the robustness against outliers in the observations, and higher efficiency in online computation compared with the kernel-based regularization method due to the recursive nature.
\end{abstract}


\section{Introduction}
Aiming at building mathematical models from observation data for complex industrial processes and natural phenomena, {\color{black}{system}} identification has received much attention from diverse research areas. Among the wide classes of mathematical models, the following stochastic system plays a basic role in the modeling of the practical systems,
\begin{align}\label{1}
y_{k+1}=\theta^{*T} x_{k} +w_{k+1},~k\geq0,
\end{align}
where $k$ is the time index, $x_k\in \mathbb{R}^{d},y_{k+1}\in \mathbb{R}^{1}$ are the observed data, $w_{k+1}\in \mathbb{R}^{1}$ is the system noise, and $\theta^*$ is an unknown parameter vector. As far as the identification of system (\ref{1}) is of concern, in the stochastic framework there has been plenty of research on the estimation of the unknown parameter vector $\theta^*$, from the classical approaches such as the least squares (LS) method, the maximum likelihood (ML) method, the instrumental variable (IV) method, and the prediction error (PE) method e.g., \cite{aastrom1971system}\cite{ljung1998system}, to the most recent progresses on the kernel-based regularization (KR) method, e.g., \cite{pillonetto2010new,chen2012estimation,pillonetto2014kernel} and the book \cite{pillonetto2022regularized} as well as the framework of the empirical risk minimization (ERM), e.g., \cite{chaudhuri2011differentially} \cite{jain2014near} \cite{jain2013differentially}.

The most classical method for the identification of system (\ref{1}) is the PE method. With a given data set $\{x_k,y_{k+1}\}_{k=1}^N$, the PE method is given by, see, e.g., \cite{ljung1998system},
\begin{align}
\widehat{\theta}_{N+1}^{PE} & =\underset{\theta}{\arg\min } R_{N+1}^{PE}(\theta), \label{2}\\
R_{N+1}^{PE}\left(\theta\right) & =\frac{1}{N} \sum_{k=1}^{N} \Phi(y_{k+1}-\widehat{y}_{k+1}(\theta,\{x_i\}_{i=1}^k)),\label{3}
\end{align}
where $\widehat{y}_{k+1}$ is a predicted value of $y_{k+1}$ based on $\{x_i\}_{i=1}^{k}$ and $\Phi(\cdot)$ is a function measuring the differences between $\widehat{y}_{k+1}$ and  ${y}_{k+1}$. A standard choice of $\Phi(\cdot)$ would be a quadratic norm $\Phi(t)=\frac{1}{2}t^2$. Other norms include $\Phi(t)=|t|$, i.e., the least absolute deviation (LAD) norm, $\Phi(t)=-\log f_w(t)$ where $f_w(\cdot)$ being the probability density function (pdf) of a stationary random sequence $\{w_k\}_{k\geq0}$ \cite{ljung1998system}, and the norm given from the minimax $M$-estimate \cite{huber2004robust}, etc. In the framework of PE method, the asymptotic properties of $R_{N+1}^{PE}$ and $\widehat{\theta}_{N+1}^{PE}$ have been studied extensively, e.g., 
\cite{lennart1999system,ljung1978convergence}. In particular, under a set of assumptions, 
the uniform convergence and almost sure (a.s.) convergence of $R_{N+1}^{PE}$, i.e., 
\begin{align*}
    \sup_{\theta\in D_{\mathcal{M}}}|R_{N}^{PE}(\theta)-R^{PE}(\theta)|\to 0, \text{a.s. as } N\to\infty,
\end{align*}
where $R^{PE}(\theta)=\lim_{N\to\infty}\mathbb{E}R_{N}^{PE}(\theta)$ and $D_{\mathcal{M}}$ is a compact set in $\mathbb{R}^d$, and the convergence and consistency of the estimator $\widehat{\theta}_{N+1}^{PE}$ were reported in \cite{lennart1999system,ljung1978convergence}. One common assumption in \cite{lennart1999system,ljung1978convergence} is that $\Phi(\cdot)$ needs to be twice differentiable and its second derivative of $\Phi(\cdot)$ always positive (more details can be found in Remark \ref{Remark6} of this paper). 
In \cite{ljung1983theory}, recursive solutions to the general criterion $R^{PE}(\theta)$ were proposed and referred to as the recursive prediction error (RPE) method, and the convergence and consistency of the RPE method was also reported. However, since the convergence analysis of the RPE method is in line with that of the ODE method, which requires the boundedness of the recursions, a project operation is often included in the RPE method (more details can be found in Remark \ref{Remark11} of this paper).

When the data set $\{x_k,y_{k+1}\}_{k=1}^N$ is short and/or has low signal-to-noise ratio, the estimate $\widehat{\theta}^{PE}_{N}$ given by the PE method may have large variance. One more effective way to deal with this issue is to   
estimate $\theta^*$ by using the KR method, e.g., \cite{pillonetto2010new,chen2012estimation,pillonetto2014kernel} and the book \cite{pillonetto2022regularized}.  
% {\color{red}For such a method, it includes two problems: kernel design and hyperparameter estimation. The former often relies on the {\em a priori} information of the underlying system, while the latter can balance bias and variance. } 
% Then the identification of $\theta^*$ falls into a Tikhonov-type variational regression problem on a suitably chosen reproducing kernel Hilbert space (RKHS). 
For the identification of system (\ref{1}), the KR method estimates $\theta^*$ by minimizing the following regularized least squares (RegLS) criterion \begin{align}
\widehat{\theta}_{N+1}^{RegLS}=\mathop{\arg\min}_{\mathcal{\theta}}{\frac{1}{N}\sum_{k=1}^{N}(y_{k+1} -\theta ^{T}x_{k})^2 + \lambda \theta^{T} P^{-1} \theta}\label{4},
\end{align}
where $\lambda>0$ is called the regularization parameter, and the matrix $P$ is symmetric and positive semi-definite and called the kernel matrix. The simulation, experimental and theoretical results in \cite{pillonetto2010new,chen2012estimation,pillonetto2014kernel,mu2018asymptotic} show that with a suitably designed kernel matrix $P$, the KR method can outperform the LS method in terms of the average accuracy and robustness \cite{pillonetto2010new,chen2012estimation,pillonetto2014kernel}  and the mean squared error (MSE) \cite{mu2018asymptotic}. 
% The KR method has been applied in the identification of both linear and nonlinear systems \cite{bottegal2016robust} \cite{bottegal2017new} \cite{risuleo2017nonparametric} \cite{khosravi2021robustness}}
%The key for such success lies in the suitably designed kernel matrix $P$. 
The design of $P$ often involves two steps:  kernel design and hyper-parameter estimation.
%, where the former is about how to embed the {\em a priori} information of the underlying system to be identified by parameterizing the kernel matrix $P$ with some hyper-parameters $\eta$, and the latter is about how to estimate the hyper-parameters $\eta$ based on the data set $\{x_k,y_{k+1}\}_{k=1}^N$. 
The key difficulty to derive a recursive algorithm for the KR method lies in a recursive solution to the hyper-parameter estimation problem, which is still an open problem, e.g., \cite{vau2023recursive}.
% The KR method also has the advantage that $\widehat{\theta}_{N+1}^{RegLS}$ in \eqref{4} has closed-form expression and moreover, re 
% A recursive algorithm to  

% To derive a recursive algorithm for the KR method contains two parts: one part for the solution of \eqref{4} and the other part to the hyper-parameter estimation, and is still an open problem, because it is hard to solve the hyper-parameter estimation problem in a recursive manner, e.g., \cite{}}

 % which has actually been applied in the identification of both linear and nonlinear systems \cite{bottegal2016robust} \cite{bottegal2017new} \cite{risuleo2017nonparametric} \cite{khosravi2021robustness}}

\begin{sloppypar}                           
It is worth noting that the identification of system (\ref{1}) is also referred to as the empirical risk minimization (ERM) in  statistics and machine learning. Define $R_{N+1}^{ERM}(\theta)=\frac{1}{N}\sum_{k=1}^{N}\Phi(y_{k+1}-\theta^{T}x_{k})$ and $R^{ERM}(\theta)=\mathbb{E}\Phi(y_{k+1}-\theta^{T}x_{k})$ provided that the mathematical expectation exists and does not depend on the time index $k$. The estimates generated by the ERM method are usually given by
\begin{align}
\widehat{\theta}_{N+1}^{ERM}=\mathop{\arg\min}_{\mathcal{\theta}}{\frac{1}{N}\sum_{k=1}^{N}\Phi(y_{k+1}-\theta^{T}x_{k}),}\label{5}
\end{align}
where the map $\Phi(\cdot)$ can penalize the deviation between linear predictors $\theta^Tx_k$
and actual outputs $y_{k+1}$. In such a framework, with given $\varepsilon>0$ and $\delta>0$, much effort, see, e.g., \cite{vapnik2000nature}, has been devoted to achieve a high
probability such that $\mathbb{P}\{|R_{N+1}^{ERM}(\widehat{\theta}_{N+1}^{ERM})-R^{ERM}(\theta_{ERM}^{*})|\leq \varepsilon\}\geq 1-\delta$, where $\theta_{ERM}^{*}=\mathop{\arg\min}_{\mathcal{\theta}} R(\theta)$, and to design numerical algorithms for solving (\ref{5}), which are expected to be scalable with the dimension of $\theta$ and the sample size $N$. 
The convergence analysis of the ERM method often relies on some widely applied technical assumptions, e.g., the Lipschitz continuity of $\Phi(y_{k+1}-\theta^Tx_k)$ with respect to $\theta$ for any data point $\{y_{k+1},x_k\}$.
%{\color{red}For theoretical analysis of the ERM method, widely applied technical assumptions include} the observations $\{x_k,y_{k+1}\}_{k\geq0}$ being a sequence of independent and identically distributed (i.i.d.) random signals, the {\color{black}{Lipschitz}} continuity of {\color{red}$\Phi(y_{k+1}-\theta^Tx_k)$ with respect to $\theta$ for any data point $\{y_{k+1},x_k\}$}, and further the {\em a priori} probability distribution function on $\{x_k,y_{k+1}\}_{k\geq0}$. 
See, e.g., \cite{chaudhuri2011differentially} \cite{jain2014near} \cite{jain2013differentially}\cite{murata2017doubly}\cite{schmidt2017minimizing} and references therein. Note that the ERM method is offline and nonrecursive in nature, i.e., when new observations are available, we need to re-optimize the empirical risk function to obtain new estimates.
\end{sloppypar}

Notice that the choice of $\Phi(\cdot)$ is of significance, which has a great influence on the estimate quality such as the effectiveness of the algorithm, robustness and etc. However, there is currently no systematic guidance available for selecting the optimal function $\Phi(\cdot)$ for a particular identification task in practice. Thus it is necessary to investigate the property of the estimate corresponding with the different selection of $\Phi(\cdot)$.
For the above-mentioned methods, they mainly focus on the asymptotic behavior of the minimizer of the optimization-based criterion function.  
Nevertheless, the technical assumptions such as the the second-order differentiability of $\Phi$ with positive values, the Lipschitz continuity of $\Phi(y_{k+1}-\theta^Tx_k)$ with respect to $\theta$ for any data point $\{y_{k+1},x_k\}$, 
etc., are restrictive, since they exclude some classical settings such as the $L_l=|\cdot|^l,~l=1, l>2$ norms and Quantile function, %the autoregressive systems with exogenous inputs (ARX), 
etc. For online identification, the RPE method exhibits certain weaknesses, such as the need for {\em {a priori}} boundedness of recursion.
Motivated by this, here, we introduce a general convex optimization-based criterion function and recursive algorithms for estimating the unknown parameter vector $\theta^*$ of the system (\ref{1}).
\begin{itemize}
\item First, we prove that the minimizer of the general convex optimization criterion converges to the true parameter vector with probability $1$.

\item Second, we show that the framework includes not only the classical LS estimator, but also the $L_l, l\geq 1$, Huber, Log-cosh, and Quantile costs for parameter identification as special cases.

\item Third, we further introduce recursive algorithms that exhibit almost sure convergence to $\theta^*$ for solving the convex optimization-based identification problem.

%\item Numerical examples are given justifying the performance of the proposed algorithms, including the strong consistency of the estimates, the robustness against outliers in the observations, and higher efficiency in online computation compared with the RegLS method due to the recursive nature.

\end{itemize}

Noting that the criterion introduced in this paper is a general convex function, there is no closed-form solution for the estimates. For this reason, the asymptotic convergence is established by analyzing the asymptotic behavior of the convex criterion function, from which the almost sure convergence of the estimates to the true vector parameter $\theta^*$ follows.  
The advantage of the proposed framework is that the convex optimization-based identification problem can be solved by recursive algorithms for the online case. Our recursive algorithm is one of stochastic approximation algorithm with expanding truncations, making the boundedness of the recursion unnecessary prior. % before proving the convergence of the recursion.
%{\color{red}Compared to RPE method \cite{ljung1983theory}, the boundedness of the recursion is not necessary before proving the convergence of the recursion by the analysis tool of stochastic approximation with expanding truncations. %the boundedness of the recursion are not necessary here by the analysis tool of stochastic approximation with expanding truncations. 
%Besides, our recursive version does not involve matrix inversion. 
As far as the authors know, the analysis method used for establishing the almost sure convergence of recursive algorithm in solving the general convex optimization-based identification problem is new, as it does not rely on the ODE approach.

%{\color{black}{Note that \cite{zhao2020general} develops a general optimization framework for the nonparametric identification of nonlinear stochastic systems, i.e., the function value at a fixed point is estimated and the optimization criterion is univariate. Different from \cite{zhao2020general}, the identification criteria in this paper are multivariate and correspondingly the theoretical analysis is more complicated, since the monotonicity of the univariate function does not hold true in the multivariate situation. }}

The remainder of this paper is organized as follows. The convex optimization-based identification criteria and technical assumptions are introduced in Section \ref{sec2}, followed by the asymptotic properties of the estimates. The recursive algorithms with almost sure convergence are established in Section \ref{sec3}. Numerical examples are given in Section \ref{sec4}, and some concluding remarks are provided in Section \ref{sec5}. Some technical lemmas and the detailed proofs of some results are given in the Appendix.

Notation: Let $\{\Omega, \mathcal{F}, \mathbb{P}\}$ be a basic probability space and $\omega$ be an element in $\Omega$. Denote by $\mathbb{E}(\cdot)$ and $\mathbb{E}(\cdot \mid \mathcal{G})$ the mathematical expectation and the conditional expectation with respect to the $\sigma$-algebra $\mathcal{G}$, respectively. Denote by $\operatorname{sgn}(x)$ the sign function, i.e., $\operatorname{sgn}(x)=1$ if $x \geq 0$ and $\operatorname{sgn}(x)=-1$ if $x<0$. The sub-gradient of a function $F(\cdot)$ is denoted by $\partial F(\cdot)$ and the gradient of $F(\cdot)$ is denoted by $\nabla F(\cdot)$, if it exists. Denote $\mathbb{I}_{A}(\cdot)$ as the indicator function on a set $A$, i.e., $\mathbb{I}_{A}(x)=1$ if $ x \in A$;  otherwise, $\mathbb{I}_{A}(x)=0$.
For two positive sequences $\left\{a_{k}\right\}_{k \geq 1}$ and $\left\{b_{k}\right\}_{k \geq 1}$, by $a_{k}=O\left(b_{k}\right)$ we mean $a_{k} \leq c b_{k}, k \geq 1$ for some $c>0$, while $a_{k}=o\left(b_{k}\right)$ means $a_{k} / b_{k} \rightarrow 0$ as $k \rightarrow \infty$.
\section{Convex Optimization-Based Identification Criteria}\label{sec2}
Let us consider the identification of the unknown parameter vector $\theta^*$ of the system (\ref{1}). With a positive convex function $\Phi(\cdot)$ and the set of observations $\left\{x_{k}, y_{k+1}\right\}_{k=1}^{N}$, the criterion function for the identification of $\theta^{*}$ is given by
\begin{align}\label{RN}
R_{N+1}(\theta)=\frac{1}{N}\sum_{k=1}^{N}\Phi(y_{k+1} -\theta ^{T}x_{k}).
\end{align}
The idea is that $\widehat\theta_{N+1}$, which minimizes the above-mentioned convex loss, i.e.,
\begin{align}\label{7}
\widehat\theta_{N+1}=\mathop{\arg\min}_{\theta\in \mathbb{R}^d}{R_{N+1}(\theta)}
\end{align}
will be a good estimate of  $\theta^{*}$ .	
	
\begin{remark}
In addition to the $L_l,l\geq1$ functions, the convex functions $\Phi(\cdot)$ widely applied in systems and control, statistics, and machine learning include the Huber function, the Log-cosh, and the Quantile functions, etc. See, e.g., \cite{wang2022comprehensive}.
\end{remark}

\begin{algorithm}\label{Algm1}
\SetAlgoNoLine
\BlankLine		
Input: $\{x_{k},y_{k+1}\}_{1\le k \le N} $

Compute the estimate
$\widehat\theta_{N+1}=\underset{\theta}{\arg\min } \frac{1}{N} \sum_{k=1}^{N} \Phi(y_{k+1}-\theta^{T}x_{k})$
as an estimate of $\theta^*$.

\caption{Convex Optimization-Based Identification Criteria}
\end{algorithm}

\begin{table}[htbp]
\caption{Class of Convex Functions of $\Phi(\cdot)$}\label{tab1}
\centering
\scalebox{0.9}{
\begin{tabular}{ll}
\toprule
$L_{l}$ function & $\Phi(x)=|x|^{l},~l\geq1$  \\
Huber function & $\Phi_{\delta}(x)=\left\{\begin{array}{ll}
\frac{1}{2}|x|^{2}, & \text { if } \quad|x| \leq \delta \\
\delta|x|-\frac{1}{2} \delta^{2}, & \text { otherwise }
\end{array}, ~~ \delta>0\right.$ \\
Log-cosh function & $ \Phi\left(x \right)=\log(\cosh(x))$ \\
Quantile function & \makecell[l]{$\Phi_{\gamma}(x)=\left\{\begin{array}{ll}				(\gamma-1)x, & \text { if } \quad {\color{black}x< 0} \\
\gamma x, & \text { if } \quad {\color{black}x\geq 0}
\end{array}, \right.$~~$\gamma \in(0,1)$   } \\
\bottomrule 	
\end{tabular}}
\end{table}

To analyze the asymptotical properties of $\widehat\theta_{N+1}$, we introduce the following assumptions on the system (\ref{1}) and the function $\Phi(\cdot)$.

\begin{itemize}
\item[A1)] The function $\Phi(\cdot)$ is nonnegative and {\color{black}{convex}} on $\mathbb{R}$, and $\Phi(t)>\Phi(0), \forall t \neq 0$. For some $c>0$ and $l>0$, $|\Phi(t)| \leq c\left(|t|^{l}+1\right)$, $\forall~t \in \mathbb{R}$. 
	
\item[A2)] The sequence $\left\{x_{k}\right\}_{k \geq 0}$ {\color{black}has a time-invariant} probability density function (pdf) $q(\cdot)$ and {there is $r>0$ such that $q(\cdot)$ is positive over the ball $B_r(0)\subset\mathbb{R}^d$ centered at the origin.} Further, $\mathbb{E}x_k=0$ and $\left\{x_{k}\right\}_{k \geq 0}$ is $\phi$-mixing with mixing coefficients $\left\{\phi_{k}\right\}_{k \geq 0}$ satisfying $\phi_{k} \le \mu \rho^{k}, k \geq 0$ for some $\mu>0$ and $0<\rho<1$, and $\mathbb{E} \|x_{k}\|^{2(l+2)+\epsilon_0} < \infty$ for some $\epsilon_0>0$, where the positive number $l$ is specified in A1).

\item[A3)] $\left\{w_{k+1}\right\}_{k \geq 0}$ is a sequence of i.i.d. random variables with a pdf denoted by $f_{w}(\cdot)$ and $f_{w}(\cdot)$ is an even function and is positive and continuous at the origin. Moreover, $w_{k+1}$ is independent of $x_{k}$ for each $k \geq 0$ and $\mathbb{E} |w_{k+1}|^{2l+\epsilon_0} < \infty $ for some $\epsilon_0>0$, where the positive number $l$ is specified in A1).
%\item[A3)] $\left\{w_{k+1}\right\}_{k \geq 0}$ is a sequence of i.i.d. random variables with a pdf denoted by $f_{w}(\cdot)$ and $f_{w}(\cdot)$ is an even function and is positive and continuous at the origin.

%\item[A4)] $w_{k+1}$ is independent of $x_{k}$ for each $k \geq 0$ and $\mathbb{E} |w_{k+1}|^{2l+{\color{red}\epsilon_0}} < \infty $, $\mathbb{E} \|x_{k}\|^{2(l+2)+{\color{red}\epsilon_0}} < \infty$ for some ${\color{red}\epsilon_0}>0$, where the positive {\color{red}number} $l$ is specified in A1).
\end{itemize}

\begin{remark}\label{r.v.moment}
The definition of $\phi$-mixing is given in Appendix A. The mixing implies that for the observations $x_k$ and $x_{k+h}$, they are asymptotically mutually independent as the time interval $h$ increases. For a linear system, the mixing assumption holds true under certain structure conditions on the system as well as some excitation conditions on the input signals, see, e.g., \cite{zhao2010recursive}.
\end{remark}

\begin{remark}\label{Remk3}
Noting that $\lim_{x \to +\infty } \frac{\log(\cosh(x))}{x}=1$, the convex functions listed in Table 1 satisfy the assumption A1). Note that $\Phi(\cdot)$ is convex and $|\Phi(t)| \le c\left(|t|^{l}+1\right)$. By using $\Phi(y)\ge \Phi(t)+\partial\Phi(t)(y-t)$, it is direct to prove that $|\partial\Phi(t)| \le c\left(|t|^{l}+1\right)$ and $|\Phi(t_2)-\Phi(t_1)|\le (|\partial\Phi(t_1)|+|\partial\Phi(t_2)|)|t_2-t_1|$. %Besides, $\varphi(\cdot)$ is odd since $\Phi(\cdot)$ is symmetric about the $y$-axis.
\end{remark}

\begin{remark}
Notice that \cite{zhao2020general} deals with the univariate convex optimization-based problem. Here we consider the multivariate optimization-based criterion, for $\theta^*\in\mathbb{R}^d$.  It is worth noting that the extension from univariate optimization problem to the multivariate optimization problem is highly nontrivial since the monotonicity of the univariate function does not hold true in the multivariate situation. Furthermore, we employ different techniques for establishing the asymptotic properties of the minimum of optimization-based criterion compared to the equation (19) in \cite{zhao2020general}. Subsequently, the theoretical analysis is more complicated.
\end{remark}

%\begin{remark}\label{Phi-Phi}
%By the relation that $\Phi(t_2)\ge \Phi(t_1)+\partial\Phi(t_1)(t_2-t_1)$, we can prove that $|\Phi(t_2)-\Phi(t_1)|\le (|\partial\Phi(t_1)|+|\partial\Phi(t_2)|)|t_2-t_1|$. In fact, if $\Phi(t_1)\ge\Phi(t_2)$, then we have $\Phi(t_1)-\Phi(t_2)\le -\partial\Phi(t_1)(t_2-t_1) \le |\partial\Phi(t_1)||t_2-t_1|$. Otherwise, if $\Phi(t_2)\ge\Phi(t_1)$, then we can obtain that $\Phi(t_2)-\Phi(t_1)\le  -\partial\Phi(t_2)(t_1-t_2) \le |\partial\Phi(t_2)||t_2-t_1|$ utilising the relation that $\Phi(t_1)\ge \Phi(t_2)+\partial\Phi(t_2)(t_1-t_2)$. Thus it follows that $|\Phi(t_2)-\Phi(t_1)|\le (|\partial\Phi(t_1)|+|\partial\Phi(t_2)|)|t_2-t_1|$.
%\end{remark}

By assumptions A1)--A3), we can define the following time-invariant function
\begin{align}\label{ExpectationRisk}
R(\theta)\triangleq \mathbb{E}(\Phi(y_{k+1} -\theta^{T}x_{k})).
\end{align}
For $R(\theta)$, we have the following results.

\begin{theorem}\label{con}
If assumptions A1), A2), and A3) hold, then $R(\theta)$ is convex with respect to $\theta$ and $\theta^{*}$ is its unique minimizer.
\end{theorem}

\begin{proof}
Please see the appendix.
\end{proof}

In practice, the expectation risk function $R(\theta)=\mathbb{E}(\Phi(y_{k+1} -\theta^{T}x_{k}))$ is unknown, since we only have a sequence of observations $\{x_k,y_{k+1}\}_{k\geq0}$. The connection between the empirical risk function $R_N(\theta)$ and {\color{black}the expectation risk }$R(\theta)$ is given as follows.

We first have a technical lemma.

\begin{lemma}\label{Nlim0}
Assume that A2)--A3) hold. Then
\begin{align}
&\frac{1}{N} \sum_{k=1}^{N} \left(\|x_k\|\|y_{k+1}\|^l-\mathbb{E}\|x_k\|\|y_{k+1}\|^l \right) \underset{N\to \infty}{\longrightarrow} 0 \text{ a.s.}\label{e1N}\\
&\frac{1}{N} \sum_{k=1}^{N} \left(\|x_k\|^{l+1}-\mathbb{E}\|x_k\|^{l+1} \right) \underset{N\to \infty}{\longrightarrow} 0 \text{ a.s.}\label{e2N}\\
&\frac{1}{N} \sum_{k=1}^{N} \left(\|x_k\|-\mathbb{E}\|x_k\| \right) \underset{N\to \infty}{\longrightarrow} 0 \text{ a.s.}\label{xN}
\end{align}
where $l>0$ in \eqref{e1N} and \eqref{e2N} is specified by A1).
\end{lemma}

\begin{proof}
Noting the mixing assumption of $\{x_k\}_{k\geq0}$ in A2), the results in fact establish the law of large numbers for the mixing process $\{x_k\}_{k\geq0}$. The detailed proof is given in the Appendix.
\end{proof}

\begin{theorem}\label{The2}
If assumptions A1), A2), and A3) %A3), and A4) 
hold, then for any fixed $\theta\in\mathbb{R}^d$,
\begin{align}\label{RNR}
R_{N}(\theta) \underset{N \rightarrow \infty}{\longrightarrow} R(\theta)~~{\color{black}\text { a.s. }}
\end{align}
Moreover, with respect to the variable $\theta$, $R_{N}(\theta)$ converges to $R(\theta)$ uniformly in any compact set $\mathcal{M}\subset\mathbb{R}^{d}$ and for $\widehat\theta_N$ generated by (\ref{7}),
\begin{align}\label{minlim}
\widehat\theta_N \underset{N \rightarrow \infty}{\longrightarrow} \theta^{*} \text { a.s. }
\end{align}
\end{theorem}

\begin{proof}
We first prove $R_{N}(\theta) \underset{N \rightarrow \infty}{\longrightarrow} R(\theta) $ almost surely, for which the analysis is similar to Lemma \ref{Nlim0}. The key step is towards establishing (\ref{minlim}), for which, in addition to (\ref{RNR}), we need a stronger result, i.e., $R_{N}(\theta)$ converges to $R(\theta)$ uniformly in any compact set $\mathcal{M}\subset\mathbb{R}^{d}$. This is resorted to a classical result in system identification, see, e.g., Lemma \ref{Lj} in the Appendix or Theorem 8.2 in \cite{lennart1999system}. The detailed proof is given in the Appendix.
\end{proof}

\begin{remark}\label{Remark6}
It is worth to note that \cite{lennart1999system,ljung1978convergence} also studied the asymptotic properties with the criterion function (\ref{RN}) for different choices of $\Phi(\cdot)$ in the framework of PE method. Focusing on the ARX model, our asymptotic results align with the ones in \cite{lennart1999system,ljung1978convergence}, e.g., the convergence (Theorem 8.2) and the consistency (Theorem 8.5) of $\widehat{\theta}_{N}^{PE}$ in \cite{lennart1999system}, but under a different set of assumptions as detailed below:
% However, different assumptions for the model and data are adopted in \cite{lennart1999system}, which are listed as follows.
\begin{itemize}
\item In \cite{lennart1999system}, the ARX model was considered and the involved transfer functions were assumed to be uniformly stable and globally identifiable, and in contrast, the linear regression model \eqref{1} is considered here, and thus the above assumptions are not needed.

\item  In \cite{lennart1999system}, the input signal $\{u_k\}$ and the output signal $\{y_k\}$ were assumed to be jointly quasi-stationary, and the measurement noise is assumed to be with bounded moments of order $4+\delta$ for some $\delta>0$. In contrast, as can be seen in A2) and A3), the regressor $\{x_k\}$ and the measurement noise $\{w_k\}$ are assumed to satisfy different assumptions and in particular, 
$\{x_k\}$ is assumed to be mixing with geometric mixing coefficients, and $\{w_k\}$ is assumed to have $\mathbb{E} |w_{k+1}|^{2l+\epsilon_0} < \infty $ for some $\epsilon_0>0$.

\item In \cite{lennart1999system}, the data set $\{u_1,y_1,u_2,y_2,\ldots\}$ was assumed to be informative to ensure the uniqueness of the minimizer of $R^{PE}(\theta)$, and in contrast, the positiveness of the pdf of $\{x_k\}$ in $B_r(0)$ in A2) is assumed to  ensure the uniqueness of the minimizer of $R(\theta)$.

\item In \cite{lennart1999system}, $\Phi(\cdot)$ is assumed to be twice differentiable and the second-order derivative is always positive, and in contrast, $\Phi(\cdot)$ is assumed to be not necessarily differentiable, as can be seen in A1). 
\end{itemize}
%In addition, \cite{ljung1978convergence} also provided convergence and consistency results for the ARX model under different assumptions on the dataset and the function $\Phi(\cdot)$, e.g., the wide-sense persistent excitation for the input signal and the linear growth rate of the derivative of $\Phi(\cdot)$.
\end{remark}

\begin{remark}
Algorithms \ref{Algm1} establish the convex optimization-based identification criteria for the system (\ref{1}). Note that Algorithms \ref{Algm1} are offline and nonrecursive in nature, i.e., when new observations are available, we need to re-optimize the criteria to obtain the new estimates. This is time-consuming for the online identification. In the next section, we will introduce the recursive algorithms with almost sure convergence to $\theta^*$ for solving the convex optimization-based identification problem.
\end{remark}

\section{Recursive Identification Algorithms}\label{sec3}
%In this section, we proceed to utilise the sequences of observations $\{x_k,y_{k+1}\}_{k\ge 1}$ to design algorithms for recursive identification. 
For the convenience of theoretical analysis, we consider the function $\Phi(\cdot)$ with a continuous derivative in part A, and $\Phi(\cdot)$ being $L_1$ function as well as Quantile function in part B, respectively. 
\subsection{The function $\Phi(\cdot)$ with a continuous derivative $\varphi(\cdot)$}

We strengthen the assumption A1) as follows.

\begin{itemize}
\item[A1')] The function $\Phi(\cdot)$ is nonnegative and {\color{black}{convex}} on $\mathbb{R}$, and $\Phi(t)>\Phi(0), ~\forall t \neq 0$. For some $c>0$ and $l>0$, $|\Phi(t)| \leq c\left(|t|^{l}+1\right)$, $\forall~t \in \mathbb{R}$. Furthermore, the first order derivative of $\Phi(\cdot)$ exists, denoted by $\varphi(\cdot)$, which is continuous on $\mathbb{R}$ and continuously differentiable on $(-\infty, 0) \cup (0, \infty)$ and for some $a>0$, $|\varphi^{(1)}(t)|\le c(|t|^l+1),~\forall~t\in (-\infty, a) \cup (a, \infty)$.
\end{itemize}

\begin{proposition}\label{Prop1}
Assume that A1'), A2), and A3) hold. Then $\nabla R(\theta)=-\mathbb{E}\left[\varphi\left(y_{k+1}-\theta^{T} x_{k}\right) \cdot x_{k}\right]$ and $\theta^{*}$ is the unique zero of the time-invariant function $\mathbb{E}\left[\varphi\left(y_{k+1}-\theta^{T} x_{k}\right) \cdot x_{k}\right]$ with respect to the variable $\theta$.
\end{proposition}

\begin{proof}
The proof is given in the Appendix.
\end{proof}

\begin{sloppypar}
\begin{remark}
By Theorem \ref{con} and Proposition \ref{Prop1}, the identification of $\theta^*$ can be transformed into the minimization of $\mathbb{E}(\Phi(y_{k+1} -\theta^{T}x_{k}))$, or equivalently, the root-searching of $\mathbb{E}\left[\varphi\left(y_{k+1}-\theta^{T} x_{k}\right) \cdot x_{k}\right]=0$ with respect to $\theta$. However, neither the function $\mathbb{E}(\Phi(y_{k+1} -\theta^{T}x_{k}))$ nor $\mathbb{E}\left[\varphi\left(y_{k+1}-\theta^{T} x_{k}\right) \cdot x_{k}\right]$ is available for identification. This motivates us to introduce the following recursive identification algorithm (Algorithm \ref{Algm3}) by using the sequence of observations $\left\{x_{k}, y_{k+1}\right\}_{k \ge 1}$.
\end{remark}
\end{sloppypar}

\begin{algorithm}\label{Algm3}	
\SetAlgoNoLine
\BlankLine		
Select an arbitrary initial value $\theta_{1}$, \\
Choose $\left\{a_{k}=\frac{1}{k}\right\}_{k \geq 1}$ and $\left\{M_{k}=k^{\frac{1}{1+2 l}}\right\}_{k \geq 1}$ with $l>0$ specified in A1'),\\
For $k\geq1$, recursively compute the estimates
\begin{subequations}
\begin{align}		
\theta_{k+1}=&\left[\theta_{k}+a_{k}x_{k} \varphi\left(y_{k+1}-\theta_{k}^{T} x_{k}\right) \right]\nonumber\\
&\cdot\mathbb{I}_{\left[\left\|\theta_{k}+a_{k} x_{k} \varphi\left(y_{k+1}-\theta_{k}^{T} x_{k}\right) \right\| \le M_{\sigma_{k}}\right]} , \label{Za} \\
\sigma_{k+1}=&\sum_{i=1}^{k}\mathbb{I}_{ \left[\left\|\theta_{i}+a_{i}x_{i} \varphi\left(y_{i+1}-\theta_{i}^{T} x_{i}\right)  \right\|>M_{\sigma_{i}} \right]}, \sigma_{1}=1. \label{Zb} 	
\end{align}
\end{subequations}
\caption{Recursive Algorithm for Function $\Phi(\cdot)$ with a Continuous Derivative $\varphi(\cdot)$}
\end{algorithm}
% \begin{remark}
% For the case that $\Phi(\cdot)=|\cdot|^2$, it is natural to compare Algorithm \ref{Algm3} with the classical recursive least squares (RLS) algorithm to \eqref{3}, which is formulated as follows:
% \begin{align}
% \theta_{k+1}&=\theta_k+a_kP_kx_k(y_{k+1}-\theta_k^Tx_k),\\
% P_{k+1}&=P_k-a_kP_kx_kx_k^TP_k,\label{19b}\\
% a_k&=(1+x_k^TP_kx_k)^{-1}.
% \end{align}
% Clearly, in terms of computation complexity, Algorithm \ref{Algm3} is much cheaper than the RLS algorithm \eqref{19b}, because the matrix-vector product $P_kx_k$ is involved at each recursion in \eqref{19b}. %In addition, it is also worth to mention that the RLS algorithm is derived from the empirical risk $R_N(\theta)$, but Algorithm \ref{Algm3} from the expectation risk $R(\theta)$.  
% \end{remark}

\begin{remark}
The algorithm (\ref{Za})--(\ref{Zb}) is a special case of stochastic approximation algorithm with expanding truncations (SAAWET, \cite{chen2014recursive}). We note that $a_k$ is the stepsize, $\mathbb{I}_{[\cdot]}$ is the indicator function, the integer $\sigma_{k}$ is the number of random truncations up to time $k$ and $M_{\sigma_{k}}$ serves as the truncation bound when the $(k+1)$-th estimate is generated. Set
\begin{align}
&F\left(\theta_{k}\right)\triangleq\mathbb{E}\left[x_{k} \varphi\left(y_{k+1}-\theta^{T} x_{k}\right)\right]\Big|_{\theta =\theta _{k}},\label{19'}\\
&\varepsilon_{k+1}\triangleq x_{k} \varphi\left(y_{k+1}-\theta_{k}^{T} x_{k}\right)-\!\mathbb{E}\left[x_{k} \varphi\left(y_{k+1}-\theta^{T} x_{k}\right)\right]\Big|_{\theta =\theta _{k}}.\label{20'}
\end{align}
Then (\ref{Za}) can be rewritten as
\begin{align}
\theta_{k+1}=&\left[\theta_{k}+a_{k}\left(F\left(\theta_{k}\right)+\varepsilon_{k+1}\right) \right]\nonumber\\ &\cdot\mathbb{I}_{\left[\left\|\theta_{k}+a_{k}\left(F\left(\theta_{k}\right)+\varepsilon_{k+1}\right) \right\| \leqslant M_{\sigma_{k}}\right]},
\end{align}
which recursively finds the root of $\mathbb{E}\left[\varphi\left(y_{k+1}-\theta^{T} x_{k}\right) \cdot x_{k}\right]$ based on the noisy observations $F\left(\theta_{k}\right)+\varepsilon_{k+1}$. Note that the analysis of the noise $\varepsilon_{k+1}$ is more complicated compared with the classical Robins-Monro (RM) algorithm, since it depends on the past estimate $\theta_k$ for $\theta^*$. Using the expanding truncation bounds, the conditions required for convergence of SAAWET are significantly weaker compared with those for the RM algorithm. See Chapters 1 and 2 in \cite{chen2014recursive}.
\end{remark}

We now analyze the convergence of Algorithm \ref{Algm3}. By Theorem \ref{ThmA1} given in the Appendix, for the convergence of $\{\theta_{k}\}_{k\geq1}$, it suffices to verify the noise condition on $\{\varepsilon_{n_k+1}\}_{k\ge1}$ along $\{n_k\}_{k\ge 1}$ of any convergent subsequence $\{\theta_{n_k}\}_{k\ge 1}$ of $\{\theta_{k}\}_{k\ge1}$ and to construct a Lyapunov function for $F(\cdot)$. The key steps towards establishing the convergence of $\{\theta_{k}\}_{k\geq1}$ are as follows.
\begin{itemize}
\item[$\bullet$] Prove that for any convergent subsequence $\{\theta_{n_{k}}\}_{k\ge1}$, there is no truncation of $\theta_{i},~i= n_k,\ldots,m(n_k,T)$ for all sufficiently large $k$ and $T>0$ small enough, where $m(n_k,T)=\max\{j:\sum_{i=n_k}^{j}a_i\le T\}$.

\item[$\bullet$] Prove that for the noise $\left\{\varepsilon_{k}\right\}_{k \geq 1}$,
$$\lim _{T \rightarrow 0} \limsup _{k \rightarrow \infty} \frac{1}{T}\Big|\sum_{i=n_{k}}^{m\left(n_{k}, T\right)} a_{i} \varepsilon_{i+1}\Big|=0$$
along the index sequence $\left\{n_{k}\right\}_{k \ge 0}$ of any convergent subsequence $\left\{\theta_{n_{k}}\right\}_{k \ge 1}$ of $\left\{\theta_{k}\right\}_{k \ge 1}$ and construct a Lyapunov function $V(\cdot)$ such that $\nabla V(u)^T \cdot F(u)\leq 0$.

\end{itemize}

We first have some technical lemmas. By carrying out a similar analysis as Lemma \ref{Nlim0}, we can obtain the following result.

\begin{lemma}\label{Le1}
Under the assumptions A1'), A2), and A3), we have
\begin{align}
&\sum_{k=1}^{\infty}{\frac{1}{k^{\frac{1+l}{1+2l}}}}\left(\left\| x_{k}\right\|^{l+1}-\mathbb{E}\left\|x_{k}\right\|^{l+1}\right) <\infty \text { a.s. \label{e3} }\\
&\sum_{k=1}^{\infty}{\frac{1}{k}}\left(\left\| x_{k}\right\|^{2}\left\|y_{k+1}\right\|^{l}-\mathbb{E}\left\| x_{k}\right\|^{2}\left\|y_{k+1}\right\|^{l}\right) <\infty \text { a.s. }\label{e5} \\
&\sum_{k=1}^{\infty}{\frac{1}{k}}\left(\left\| x_{k}\right\|^{l+2}-\mathbb{E}\left\| x_{k}\right\|^{l+2}\right) <\infty \text { a.s. }\label{e6}
\end{align}
with $l>0$ specified in A1') and for any fixed $\theta \in \mathbb{R}^{d}$,
\begin{align}
\label{e4} &\sum_{k=1}^{\infty} \frac{1}{k}\left[ x_{k}\varphi\left(y_{k+1}-\theta^{T}x_{k}\right)-\mathbb{E}x_{k}\varphi\left(y_{k+1}-\theta^{T}x_{k}\right) \right] <\infty \text { a.s. }
\end{align}
\end{lemma}

The next lemma indicates that for the convergent subsequence $\left\{\theta_{n_{k}}\right\}_{k \geq 1}$ of $\left\{\theta_{k}\right\}_{k \geq 1}$, there is no truncation for $\theta_{j+1},~j=n_{k}, \ldots, m\left(n_{k}, T\right)$ if $k$ is large enough.

\begin{lemma}\label{Le2}
Assume that A1')--A3) hold and $\left\{\theta_{n_{k}}\right\}_{k \geq 1}$ is a convergent subsequence of $\left\{\theta_{k}\right\}_{k \geq 1}$ generated from Algorithm \ref{Algm3}, i.e., $\theta_{n_{k}} \underset{k \rightarrow \infty}{\longrightarrow} \bar{\theta}$. Then for any fixed $T>0$ small enough, it holds that for all $k$ sufficiently large, there is no truncation for $\theta_{j+1},~j=n_{k}, \ldots, m\left(n_{k}, T\right)$ and
\begin{align}\label{cT}
\left\|\theta_{j+1}-\theta_{n_{k}}\right\| \leq C T,~~j=n_{k}, \ldots, m\left(n_{k}, T\right)
\end{align}
where $C>0$ is a constant possibly depending on $\omega$.
\end{lemma}	

\begin{proof}
Please see the Appendix.
\end{proof}

For the noise $\varepsilon_{k+1}$ defined by (\ref{20'}), we have the following result.

\begin{lemma}\label{Lem4}
If A1'), A2), and A3) hold, then for any convergent subsequence $\left\{\theta_{n_{k}}\right\}_{k \geq 1} $ of $ \left\{\theta_{k}\right\}_{k \geq 1}$, it holds that
\begin{align}\label{noise}
\lim _{T \rightarrow 0} \limsup _{k \rightarrow \infty} \frac{1}{T}\Big\|\sum_{i=n_{k}}^{m\left(n_{k}, t_{k}\right)} a_{i} \varepsilon_{i+1}\Big\|=0,~~\forall t_{k} \in[0, T]~~\mathrm{a.s.}
\end{align}		
\end{lemma}

\begin{proof}
Please see the Appendix.
\end{proof}

The following result establishes that the recursive estimates generated from Algorithm \ref{Algm3} can converge to $\theta^*$ almost surely.

\begin{theorem}\label{Thm3}
Assume A1'), A2), and A3)
hold. Then for $\left\{\theta_k\right\}_{k \geq 1}$ generated by (\ref{Za})--(\ref{Zb}),  it holds that
\begin{align} \label{a.s.}
\theta_k \underset{k \rightarrow \infty}{\longrightarrow} \theta^{*}~~\text { a.s. }
\end{align}
\end{theorem}

\begin{proof}
By the general convergence theorem of SAAWET, i.e., Theorem \ref{ThmA1} given in the Appendix, for (\ref{a.s.}) it suffices to verify the noise condition on $\varepsilon_{k+1}$ and the stability condition on $F(\cdot)$, which are defined by (\ref{20'}) and (\ref{19'}), respectively.

By Lemma \ref{Lem4}, it follows that for $\varepsilon_{k+1}$ defined by (\ref{20'}),
\begin{align*}
\lim _{T \rightarrow 0} \limsup _{k \rightarrow \infty} \frac{1}{T}\Big\|\sum_{i=n_{k}}^{m\left(n_{k}, t_{k}\right)} a_{i} \varepsilon_{i+1}\Big\|=0,~~\forall t_{k} \in[0, T]~~\mathrm{a.s.}
\end{align*}

Define
$$
V(\theta)\triangleq \mathbb{E}\left(\Phi\left(y_{k+1}-\theta^{T} x_{k}\right)\right).
$$

By Proposition \ref{Prop1}, we have
\begin{align}
\nabla V(\theta )=-F(\theta)
\end{align}
and $\theta^{*}$ is the unique zero point of $F(\theta)=0$, and thus
\begin{align}\label{C2_1}
\sup_{\delta \le \left\|\theta-\theta^{*}\right\| \le \Delta} \nabla V(\theta)^{T} F(\theta)<0,~~\forall~0<\delta<\Delta.
\end{align}

By Theorem \ref{con}, $\theta^{*}$ is the unique minimizer of $V(\cdot)$. We can prove that there exists some $c_{0}>0$ such that
\begin{align}\label{C2_2}
V\left(0\right)<\inf_{\left\| \theta \right\|=c_{0} } V(\theta).
\end{align}

Based on the above analysis, it is seen that C1)--C4) given in the Appendix are satisfied for Algorithm \ref{Algm3}. Thus the almost sure convergence of $\left\{\theta_{k}\right\}_{k \ge 1}$ to $\theta^{*}$ follows.
\end{proof}

\subsection{The $L_1$ function $\Phi(t)=|t|$ and the Quantile function $\Phi_{\gamma}(x)$}

\begin{proposition}\label{Prop2}
Assume that A2) and A3) hold. Then the unknown parameter vector $\theta^*$ of the system (\ref{1}) is the unique zero point of $\mathbb{E}x_k\mathrm{sgn}\left(y_{k+1}-\theta^{T} x_{k}\right)=0$.
\end{proposition}

\begin{proof}
Similar to the proof of Proposition \ref{Prop1}, we can prove that $\nabla\mathbb{E}|y_{k+1}-\theta^{T} x_{k}|=-\mathbb{E}x_k\mathrm{sgn}\left(y_{k+1}-\theta^{T} x_{k}\right)$. By Theorem 1, $\theta^*$ is the unique minimizer of the convex function $\mathbb{E}|y_{k+1}-\theta^{T} x_{k}|$. Thus it holds that $\theta^*$ is the unique zero point of $\mathbb{E}x_k\mathrm{sgn}\left(y_{k+1}-\theta^{T} x_{k}\right)=0$.
\end{proof}

For $\Phi(t)=|t|$, a recursive algorithm parallel to Algorithm \ref{Algm3} is given as follows.

\begin{algorithm}\label{Algm4}
\SetAlgoNoLine 	
\BlankLine	
Select an arbitrary initial value $\theta_{1}$, \\
Choose $\left\{a_{k}=\frac{1}{k}\right\}_{k \geq 1}$ and $\left\{M_{k}\right\}_{k \geq 1}$ a sequence of positive numbers strictly diverging to infinity,\\
For $k\geq1$, recursively compute the estimates
\begin{subequations}
\begin{align}
\theta_{k+1}=&\left[\theta_{k}+a_{k}x_{k} \mathrm{sgn}\left(y_{k+1}-\theta_{k}^{T} x_{k}\right) \right]\nonumber\\
&\cdot\mathbb{I}_{\left[\left\|\theta_{k}+a_{k} x_{k} \mathrm{sgn}\left(y_{k+1}-\theta_{k}^{T} x_{k}\right) \right\| \le M_{\sigma_{k}}\right]}, \label{Aa} \\
\sigma_{k+1}=&\sum_{i=1}^{k}\mathbb{I}_{ \left[\left\|\theta_{i}+a_{i}x_{i} \mathrm{sgn}\left(y_{i+1}-\theta_{i}^{T} x_{i}\right)  \right\|>M_{\sigma_{i}}\right]},~~\sigma_{1}=1.\label{Ab}
\end{align}
\end{subequations}
\caption{Recursive Algorithm for $L_1$ Function $\Phi(t)=|t|$}
\end{algorithm}

Similar to Lemmas \ref{Le1} and \ref{Le2}, the following results hold for Algorithm \ref{Algm4}.

\begin{lemma}\label{le4}
Assume that A2) and A3) hold. Let $\left\{\theta_{n_k}\right\}_{k \geq 1}$ be any convergent subsequence of $\left\{\theta_k\right\}_{k \geq 1}$ generated from (\ref{Aa})--(\ref{Ab}), i.e., $\theta_{n_k} \underset{k \rightarrow \infty}{\longrightarrow} \bar{\theta}$. Then, for any fixed $T>0$ small enough, it holds that for all $k$ sufficiently large
\begin{align}\label{absi}
\left\|\theta_{j+1}-\theta_{n_k}\right\| \le CT,~~j=n_k, \ldots, m\left(n_k, T\right),
\end{align}
where $C>0$ is a constant not depending on $k$ or $T$.
\end{lemma}

\begin{proof}
The proof can be similarly obtained as Lemma \ref{Le2}.
\end{proof}

\begin{lemma}\label{le5}
Assume that A2), and A3) hold. Then
\begin{align}
\sum_{k=1}^{\infty} \frac{1}{k} &\left(\left\|x_{k}\right\| \mathbb{I}_{[\left| y_{k+1}-z^{T}x_{k}\right| \le \left\|x_{k}\right\| (CT+ x )]}\right.\nonumber\\
&-\left.\mathbb{E}\left\|x_{k}\right\| \mathbb{I}_{[\left| y_{k+1}-z^{T}x_{i}\right| \le \left\|x_{k}\right\| (CT+ x )]}\right)<\infty~~\text{a.s.} \label{le5_1} \\	
\sum_{k=1}^{\infty} \frac{1}{k}&[x_{ k} \operatorname{sgn}(y_{ k+1}-\theta^{T}x_{ k})] \nonumber\\
&-\mathbb{E}[x_{k} \operatorname{sgn}(y_{ k+1}-\theta^{T}x_{k})]<\infty~~\text { a.s. } \label{le5_2}		
\end{align}	
for any fixed $C>0,~T>0,~x>0,~z\in \mathbb{R}^d$ and $\theta\in \mathbb{R}^d$.
\end{lemma}

\begin{proof}
The proof can be similarly obtained as Lemma \ref{Le1}.
\end{proof}

The following result establishes the almost sure convergence of the recursive identification algorithm (\ref{Aa})--(\ref{Ab}).

\begin{theorem}\label{Thm4}
Assume that A2), and A3) hold. Then for $\{\theta_k\}_{k\geq1}$ generated from (\ref{Aa})--(\ref{Ab}), it holds that
\begin{align} \label{Thm4absa.s.}
\theta_k \underset{k \rightarrow \infty}{\longrightarrow} \theta^{*}~~\text { a.s. }
\end{align}
\end{theorem}

\begin{proof}
The proof can be similarly obtained as Theorem \ref{Thm3}. Here we only list the sketch.

The algorithm (\ref{Aa}) can be rewritten as follows
\begin{align}
\theta_{k+1}=&\left[\theta_{k}+a_{k}\left(F\left(\theta_{k}\right)+\varepsilon_{k+1}\right)\right]\nonumber\\
&\cdot\mathbb{I}_{\left[\left\|\theta_{k}+a_{k}\left(F\left(\theta_{k}\right)+\varepsilon_{k+1}\right)\right\| \leq M \sigma_{k}\right]},
\end{align}		
where
\begin{align}
& F\left(\theta_{k}\right) =\left.\mathbb{E}\left[x_{k} \mathrm{sgn}\left(y_{k+1}-\theta^{T} x_{k}\right)\right]\right|_{\theta=\theta_{k}}, \\
& \varepsilon_{k+1}\! \!=\!x_{k} \varphi\!\left(y_{k+1}\!\!-\!\theta_{k}^{T} x_{k}\!\right)\!-\!\mathbb{E}\!\left[x_{k} \varphi\!\left(y_{k+1}\!-\!\theta^{T} x_{k}\right)\!\right]\!\!\Big|_{\theta=\theta_{k}}.\label{absnoise}
\end{align}	

Define
\begin{align}
V(\theta)\triangleq \mathbb{E}|y_{k+1}-\theta^{T} x_{k}|.
\end{align}

Noting that
\begin{align}
\nabla V(\theta )=-\mathbb{E}\left[x_{k} \mathrm{sgn}\left(y_{k+1}-\theta^{T} x_{k}\right)\right],
\end{align}
similar to (\ref{C2_1}) and (\ref{C2_2}), we can prove
\begin{align}
\sup_{\delta \le \left\|\theta-\theta^{*}\right\| \le \Delta} \nabla V(\theta)^{T} F(\theta)<0,~~\forall~0<\delta<\Delta,
\end{align}
and
\begin{align}
V\left(0\right)<\inf_{\left\| \theta \right\|=c_{0} } V(\theta),~~\text{for~some}~c_{0}>0.
\end{align}

Next we consider the noise $\varepsilon_{k+1}$ defined by (\ref{absnoise}). Denote by $\left\{\theta_{n_{k}}\right\}_{k\geq1}$ a convergent subsequence of $\left\{\theta_{k}\right\}_{k\geq1}$ and by $\bar{\theta}$ the limit of $\left\{\theta_{n_{k}}\right\}$, i.e., $\theta_{n_{k}} \underset{k \to \infty}{\longrightarrow} \bar{\theta}$.

Similar to the proof of Lemma \ref{Lem4}, we can verify that
\begin{align}\label{absnoise5}
 \lim _{T \rightarrow 0} \limsup _{k \rightarrow \infty} \frac{1}{T}\Big|\sum_{i=n_{k}}^{m\left(n_{k}, T\right)} a_{i} \varepsilon_{i+1}\Big|=0~~\mathrm{a.s.}
\end{align}

Then by Theorem \ref{ThmA1} given in the Appendix, the conclusion of {\color{black}(\ref{Thm4absa.s.})} follows.
\end{proof}

\begin{algorithm}\label{Algm5}
\SetAlgoNoLine 	
\BlankLine	
Select an arbitrary initial value $\theta_{1}$, \\
Choose $\left\{a_{k}=\frac{1}{k}\right\}_{k \geq 1}$ and $\left\{M_{k}\right\}_{k \geq 1}$ a sequence of positive numbers strictly diverging to infinity,\\
For $k\geq1$ and $0<\gamma<1$, recursively compute the estimates
\begin{subequations}
\begin{align}
\theta_{k+1}\!=\!&\left[\theta_{k}\!+\!a_{k}O_{k+1} \right]\cdot\mathbb{I}_{[\|\theta_{k}+a_{k}O_{k+1} \| \le M_{\sigma_{k}}]} \label{quantilea}, \\
\sigma_{k+1}\!=\!&\sum_{i=1}^{k}\mathbb{I}_{ \left[\left\|\theta_{i}+a_{i}O_{i+1}  \right\|>M_{\sigma_{i}}\right]},
\sigma_{1}=1\label{quantileb}.
\end{align}
\end{subequations}
where $O_{i+1}=\gamma x_i\mathbb{I}_{[|y_{i+1}-\theta^{T}_i x_{i}|\geq 0]}+(\gamma-1) x_i\mathbb{I}_{[|y_{i+1}-\theta^{T}_i x_{i}|<0]}$.
\caption{Recursive Algorithm for Quantile Function $\Phi_{\gamma}(x)$}
\end{algorithm}
Similar to Proposition \ref{Prop2} and Theorem \ref{Thm4}, for Algorithm \ref{Algm5} the following results hold.

\begin{proposition}\label{Prop3}
Assume that A2) and A3) hold. Then with respect to the variable $\theta$, $\theta=\theta^*$ is the unique zero of
$$\mathbb{E}(\gamma x_k\mathbb{I}_{|y_{k+1}-\theta^{T} x_{k}|\geq 0}+(\gamma-1) x_k\mathbb{I}_{|y_{k+1}-\theta^{T} x_{k}|<0})=0.$$
\end{proposition}

\begin{theorem}\label{Thm5}
Assume that A2) and A3) hold. Then for $\{\theta_k\}_{k\geq1}$ generated from (\ref{quantilea})--(\ref{quantileb}), it holds that
\begin{align} \label{absa.s.}
\theta_k \underset{k \rightarrow \infty}{\longrightarrow} \theta^{*}~~\text { a.s. }
\end{align}
\end{theorem}

\begin{proof}
Here we only list the key points. The algorithm (\ref{quantilea})--(\ref{quantileb}) can be rewritten as follows
\begin{align}
\theta_{k+1}=&\left[\theta_{k}+\overline{a}_{k}x_{k} \mathrm{sgn}\left(y_{k+1}-\theta_{k}^{T} x_{k}\right) \right]\nonumber\\
&\cdot\mathbb{I}_{\left[\left\|\theta_{k}+\overline{a}_{k} x_{k} \mathrm{sgn}\left(y_{k+1}-\theta_{k}^{T} x_{k}\right) \right\| \le M_{\sigma_{k}}\right]}, \label{47'} \\	\sigma_{k+1}=&\sum_{i=1}^{k}\mathbb{I}_{ \left[\left\|\theta_{i}+\overline{a}_{i}x_{i} \mathrm{sgn}\left(y_{i+1}-\theta_{i}^{T} x_{i}\right)  \right\|>M_{\sigma_{i}}\right]},~~\sigma_{1}=1,\label{48'}
\end{align}
where
\begin{align}
\overline{a}_{k}=
\begin{cases}
a_k\cdot \gamma,~~\mathrm{if}~y_{k+1}-\theta_{k}^{T} x_{k}\geq0,\\
a_k\cdot (1-\gamma),~~\mathrm{if}~y_{k+1}-\theta_{k}^{T} x_{k}<0.
\end{cases}
\end{align}

Noting that $0<\gamma<1$, for $\overline{a}_k$ it follows that $\overline{a}_k\to0$ as $k\to\infty$ and $\sum_{k=1}^{\infty}\overline{a}_k=\infty$. The proof can be similarly obtained as Theorem \ref{Thm4}.
\end{proof}

\begin{remark}
Algorithms \ref{Algm3}--\ref{Algm5} establishes the recursive identification algorithms for the general convex function $\Phi(\cdot)$ including $\Phi(t)=|t|^l,~l\geq1$ and the Quantile function $\Phi_{\gamma}(t)$, etc. These algorithms can be easily calculated in an online manner, since to obtain the new estimates, we only need to update the current estimates with the new observations.
\end{remark}

\begin{remark}\label{Remark11}
Focusing on the ARX model,  \cite{ljung1983theory} also introduced a recursive algorithm for the general criterion function $R(\theta)=\mathbb{E}(\Phi(y_{k+1}-\theta^Tx_k))$, which can be formulated as follows:
\begin{align}\label{RPEM}
\theta_{k+1}^{RPEM}=\theta_k^{RPEM}+a_kR_k^{-1}x_k\varphi(y_{k+1}-\theta_k^{T}x_k),    
\end{align}
where $R_k\in \mathbb {R}^{d\times d}$ is a positive definite matrix that modifies the gradient search direction to perhaps a suitable one, while the other terms in (\ref{RPEM}) align with ones in Algorithm \ref{Algm3}--\ref{Algm5}. Clearly, in terms of computation complexity, Algorithms \ref{Algm3}-5 are much cheaper than the recursive algorithm associated with \eqref{RPEM}, because 
of the matrix inversion of $R_k$ except for being given by the properly scaled identity matrix. What's more, \cite{ljung1983theory} also studied the  convergence of the recursive algorithm associated with \eqref{RPEM} but under a different set of assumptions as detailed below:
\begin{itemize}
\item \cite{ljung1983theory} describes the model in the form of transfer function, where the model predictor is assumed to be stable, and in contrast, the linear regression model \eqref{1} is considered here.

\item For the theoretical analysis of \eqref{RPEM}, the data generation $\{y_k,u_k\}$ should be asymptotically mean stationary and exponentially stable. In contrast, we require the regressor being mixing sequences.
\item In \cite{ljung1983theory}, $\Phi(\cdot)$ is assumed to be twice differentiable and the derivative is bounded by a linear function, while the second-order derivative is assumed to be bounded. In contrast, our requirements for $\Phi(\cdot)$ is listed in A1').
\end{itemize}

The theoretical analysis of the RPE method is linked to the idea of ODE method, where the convergence of the recursion is tied to the stability of the associated differential equation. To assure a bounded sequence, the algorithm (\ref{RPEM}) includes a projection to a compact set %with {\em{prior}} knowledge 
in $\mathbb{R}^d$, which is not needed in Algorithm \ref{Algm3}-\ref{Algm5}. %Notice that \eqref{RPEM} with $R_k$ being a properly scaled identity matrix coincides with RM algorithm. This demonstrates the advantage of SAAWET in a certain sense. It is also of interest considering (\ref{RPEM}) with expanding truncation for the general search direction.
\end{remark}

\section{Numerical Examples}\label{sec4}

In this section, we testify the performance of the algorithms through numerical examples. The convex functions in the general criteria are chosen as the $L_{l},l\geq1$ function, the Huber function, the Log-cosh function, and the Quantile function. In the first example, with a fixed sample size $N$ we testify the performance of Algorithm \ref{Algm1} as well as its robustness against the outliers in the observations. In the second example, we testify the performance of the recursive estimates generated from Algorithms \ref{Algm3}, \ref{Algm4}, and \ref{Algm5}. We also compare the proposed algorithms with the regularized least squares (RegLS) algorithms with the SS, DC, TC kernels \cite{chen2018kernel}\cite{chen2011kernel}.

{\em Example 1)} Consider the following linear system:
\begin{align}
A\left(q^{-1}\right) y_k=B\left(q^{-1}\right) u_k+w_k
\end{align}
where $q^{-1}$ is the backward-shift operator, $A\left(q^{-1}\right)=1-1.5q^{-1}+0.7q^{-2}$, and $B\left(q^{-1}\right)= q^{-1}+0.5q^{-2}$. Set $\theta ^{*}=[-1.5,0.7,1,0.5]^{T}$ and $x_k=[-y_{k-1}~-y_{k-2}~u_{k-1}~u_{k-2}]^T$. Choose the input $\{u_k\}_{k\geq1}$ an i.i.d. sequence with a uniform distribution on $[-0.5,0.5]$ and assume the noise $\{w_k\}_{k\geq1} $ an i.i.d. sequence with a Gaussian distribution $\mathcal{N}(0,0.1)$. With a fixed length of data set, we use the set of observations $\left\{x_{k}, y_{k+1}\right\}_{k=1}^{N=2000}$ for identification, and the set $\left\{x_{k}, y_{k+1}\right\}_{k=2001}^{2200}$ for validation of the identification results. The identification error is characterized by
\begin{align}\nonumber
\mathrm{Error} = \Big\|\widehat\theta_{N}-\theta^{*}\Big\|.
\end{align}
We adopt the Matlab toolbox CVX to numerically solve the estimates from Algorithm \ref{Algm1}. As a comparison, we also compute the RegLS estimates by the Matlab toolbox arxRegul.

{\color{black}For the Huber function and the Quantile function, we choose $\delta=1$ and $\gamma=0.4$, denoted by $\mathrm{Huber}_1$ and $\mathrm{Q}_{0.4}$, respectively.} We perform 100 Monte Carlo simulations. In Figure \ref{errorbox}, the box plot shows the identification errors from 100 simulations. From Figure \ref{errorbox}, we can see that there exists some convex criteria perform comparable with the RegLS methods with the SS, DC, TC kernels.
%In Figure \ref{pred}, the blue lines show the true system outputs while the red lines show the predicted outputs $\widehat{y}_{k}={\color{red}\widehat\theta_N^Tx_k},~k=2001,\cdots,2200$ from one sample path of 100 simulations. 

\begin{figure}[hbpt]
\centering
\includegraphics[width=0.7\linewidth]{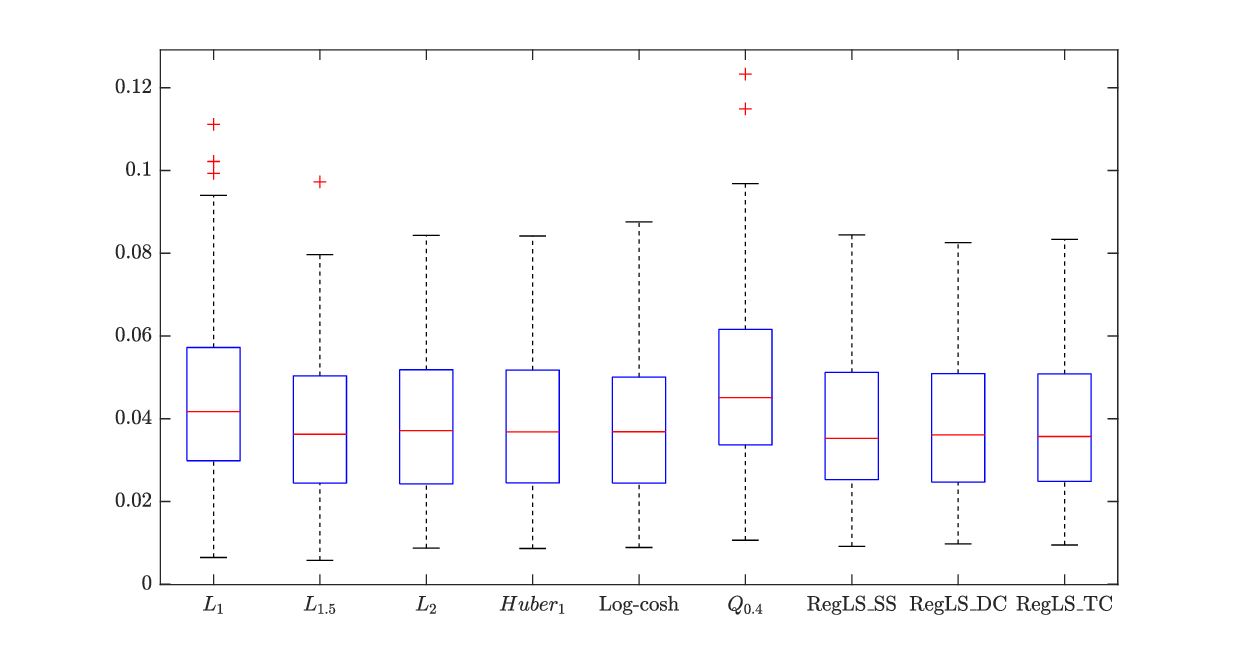}%{cvx_1_MC}
		\caption{Box plot of identification errors}
		\label{errorbox}
\end{figure}
% \begin{figure}[hbpt]
% \centering
% \includegraphics[width=12cm]{CVXMC3kernlpredicts}%{cvx_1_predicts}
% \caption{Predicted and actual outputs}
% \label{pred}
% \end{figure}

We further testify the robustness of Algorithm \ref{Algm1}.
{\color{black}For the data set $\left\{x_{k}, y_{k+1}\right\}_{k=1}^{N=2000}$, we set the fraction of outliers as $1\%$. This is achieved by replacing the observations $(x_k,y_{k+1})$ at $k=100,200,...,2000$ with the following outlier $(X_{\mathrm{otl}},y_{\mathrm{otl}}),~X_{\mathrm{otl}}\in \mathbb{R}^4,~y_{\mathrm{otl}}\in \mathbb{R}^1$, }
%We add the following outlier $(X_{\mathrm{otl}},y_{\mathrm{otl}}),~X_{\mathrm{otl}}\in \mathbb{R}^4,~y_{\mathrm{otl}}\in \mathbb{R}^1$ into the observations $(\varphi_k,y_k)$ at time instances $k=100,~200,\cdots,~2000$, {\color{blue}{(QQ??)}} where
\begin{align*}
X_{\mathrm{otl}}=\frac{1}{\alpha \cdot n} X^{T} y,~~
y_{\mathrm{otl}}=\beta,
\end{align*}
where $X=[x_{1},x_{2},\cdots,x_{2000}]^T,~y=[y_{1},y_{2},\cdots,y_{2000}]^{T}$, $n=20$, $\beta=\frac{1}{2000}\sum_{k=1}^{2000} y_{k}$, and $\alpha=15$ (see, e.g., \cite{diakonikolas2019sever}).
 \begin{figure}[hbpt]
 \centering	\includegraphics[width=0.7\linewidth]{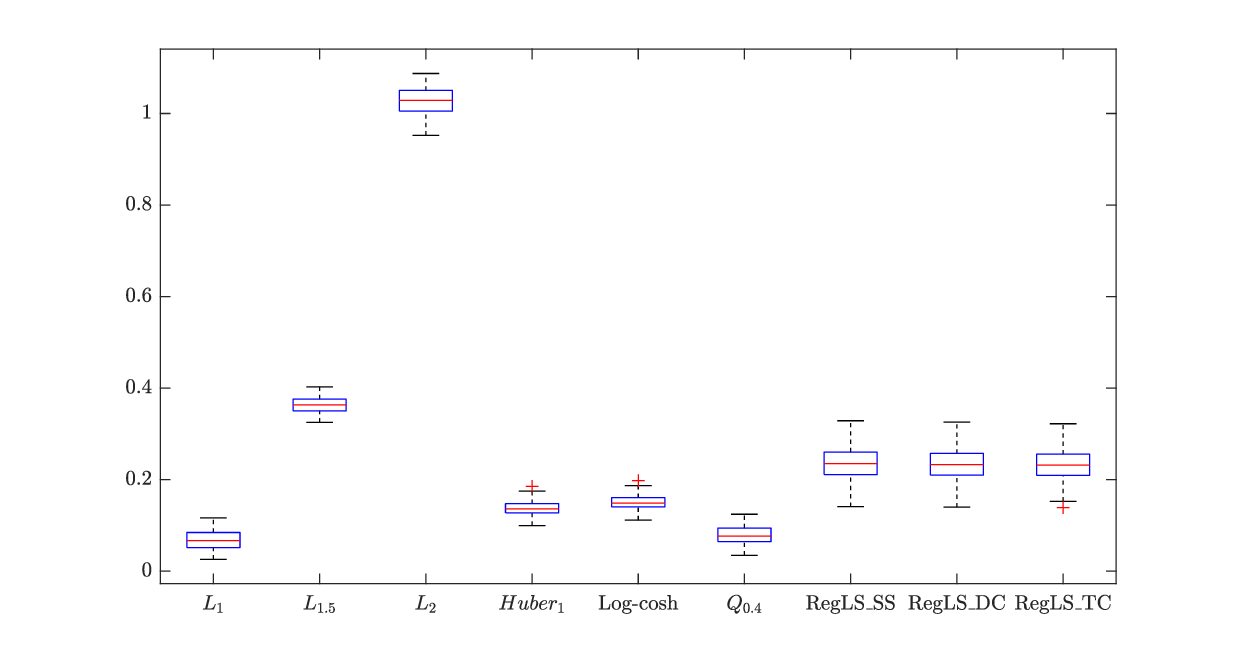}	
   \caption{Box plot of identification errors with outliers}
 		\label{errorout}
 \end{figure}
% \begin{figure}[hbpt]
% \centering	\includegraphics[width=12cm,height=6cm]{cvx_outl_MC}
% 		\caption{Box plot of identification errors with outliers}
% 		\label{errorout}
% \end{figure}
% \begin{figure}[hbpt]
%     \centering
% \includegraphics[width=12cm]{cvx_1_mc_outlier_pre1}
%     \caption{Predicted and actual outputs with outliers}
%     \label{predout}
% \end{figure}
%\begin{figure}[hbpt]
%\centering
	%\begin{minipage}[H]{0.5\linewidth}
%		\centering
%		\includegraphics[width=12cm]{cvx_outl_predicts}
%		\caption{Predicted and actual outputs with outliers}
%\label{predout}
	%\end{minipage}
%\end{figure}	

We also perform 100 Monte Carlo simulations. In Figure \ref{errorout}, the box plot shows the identification errors from 100 simulations. From Figure \ref{errorout}, we can see that there exists some convex criteria perform better than the RegLS methods.
%In Figure \ref{predout}, the blue lines show the true system outputs while the red lines show the predicted outputs $\widehat{y}_{k}={\color{red}\widehat\theta_N^Tx_k},~k=2001,\cdots,2200$ from one sample path of 100 simulations. 

{\em Example 2)} We still consider the linear system in Example 1). Let $\{u_k\}_{k\geq1}$ and $\{w_k\}_{k\geq1}$ be two sequences of mutually independent and i.i.d. random variables with distributions $u_k \in \mathcal{N}(0,1)$ and $w_k \in \mathcal{N}(0,0.1)$. We testify the performance of the recursive estimates generated from Algorithms \ref{Algm3}--\ref{Algm5} and set the maximal number of recursion steps as $10000$. We perform 100 Monte Carlo simulations and compare the recursive estimates with the RegLS methods with the SS, DC, TC kernels. %{\color{red}Notice that when the new data becomes available, the regularized methods utilise all collected data to compute the hyperparameters, followed by optimizing a regularization problem. Here we use the Matlab toolbox arxRegulOptions for calculating the regularization parameter and the kernel matrix for different choice of kernel design. For fair comparison, the time for computing the hyperparameters is omitted.}

\begin{figure}[hbpt!]
\centering
\includegraphics[width=0.7\linewidth]{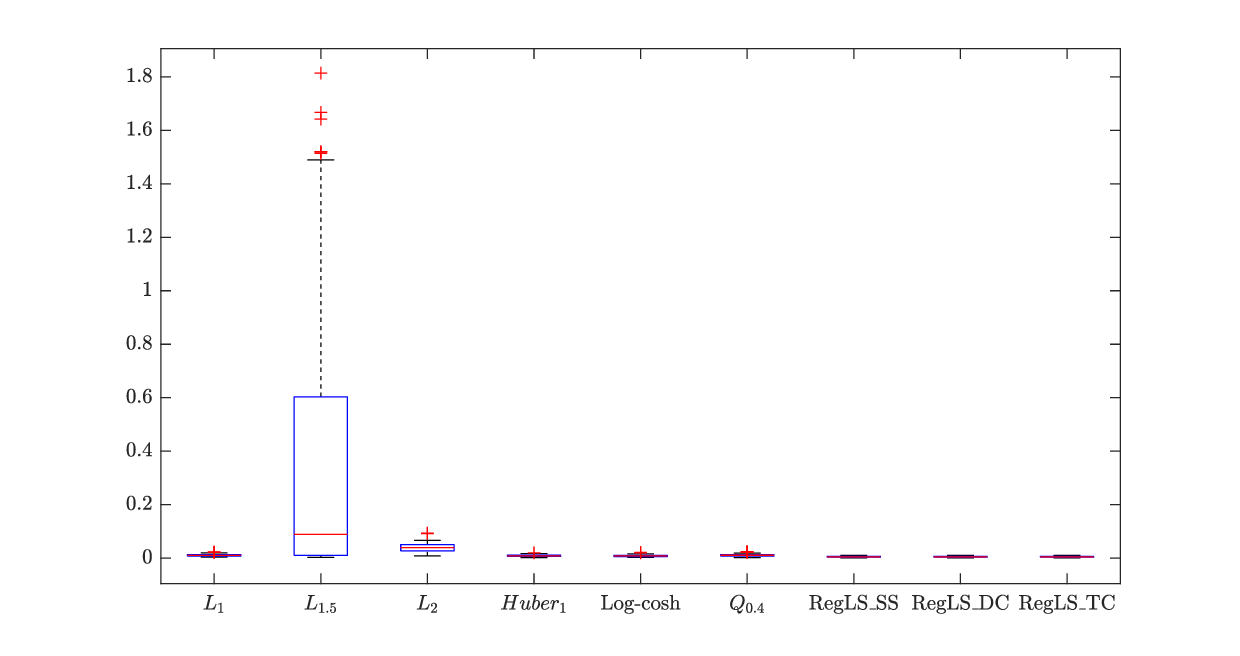}%{SAMC}
\caption{Box plot of identification errors with maximum number of recursion 10000}
\label{SAbox}
\end{figure}
\begin{figure}[hbpt!]
\centering
\includegraphics[width=0.7\linewidth]{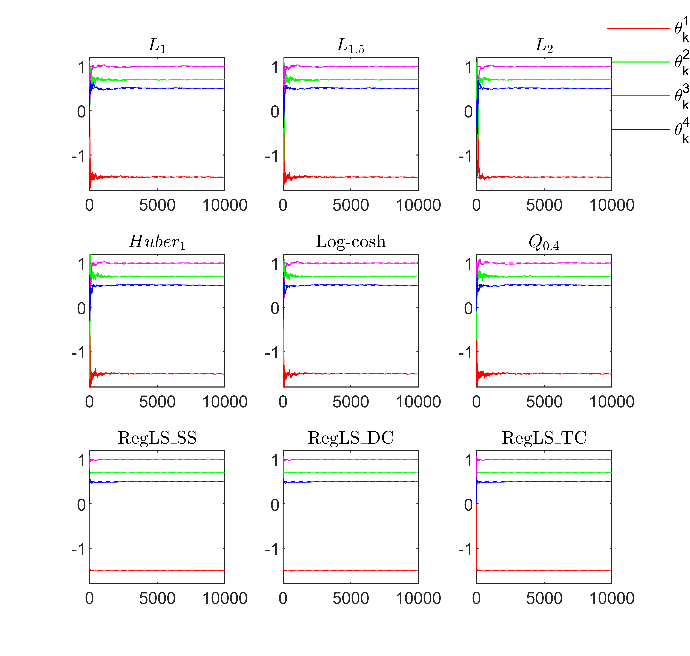}
\caption{{\color{black} Recursive identification sequences $\theta_k^{i*}$ for $\theta^{i*},i=1,2,3,4$,~$k=1,\cdots,10000$}}
\label{SAchase}
\end{figure}
\begin{table}
\setlength\tabcolsep{3pt}
\centering
\caption{Estimated values of unknown parameters at $k=10000$ and accumulated computation time from $k=1,\cdots,10000$ }
\begin{tabular}{cccccc}		
\toprule
& $\theta_{1}^{*}$ & $\theta_{2}^{*}$ & $\theta_{3}^{*}$ & $\theta_{4}^{*}$ & time (s)\\
\midrule
true value& -1.5 & 0.7 & 1 & 0.5\\

$L_{1}$ function & -1.5070 & 0.7035& 0.9946&0.4979&0.0203 \\
$L_{1.5}$ function & -1.5095 & 0.7023  &0.9944 &0.4973 &0.0293 \\
$L_{2}$ function  & -1.5062   &0.7007 &0.9947 &0.4991 &0.0299 \\
{\color{black}Huber function, $\delta=1$}	& -1.5079    &0.7011 & 0.9948 & 0.4977 &0.0520 \\
Log-cosh function & -1.5075   & 0.7010 & 0.9947 & 0.4982 &0.0225 \\
Quantile function, $\lambda=0.4$   & -1.5063   & 0.7052 & 0.9925 & 0.4984 &0.0207 \\
RegLS with SS kernel & -1.5014   & 0.7010 & 0.9937 & 0.4988 &902.4207 \\
RegLS with DC kernel & -1.5015   & 0.7010 & 0.9937 & 0.4988 &760.2699 \\	
RegLS with TC kernel& -1.5015   & 0.7010 & 0.9937 & 0.4988 &822.4108 \\	
\bottomrule	
\end{tabular}
\label{tab}
\end{table}

Figure \ref{SAbox} shows the box plot of the identification errors of 100 simulations. Figure \ref{SAchase} shows the identification sequences generated from one of the simulations, where the dashed lines denote the true values of parameters {\color{black}$\theta^{*}=[\theta^{1*},~\theta^{2*},~\theta^{3*},~\theta^{4*}]^T$} and the solid lines denote the estimates, %From Figure \ref{SAchase} we find that the simulation results are consistent with the theoretical analysis. In such one of the simulations, 
and the values of estimates at the final step $10000$ and the accumulated computation time from $k=1,\cdots,10000$ are summarized in Table \ref{tab}, which shows that for the online identification, the recursive algorithms proposed in the paper are more efficient than the classical RegLS methods.

\section{Concluding Remarks}\label{sec5}

In this paper, a general convex optimization-based criterion is introduced for estimating the unknown parameter vector $\theta^{*}$ of a linear stochastic system and the recursive algorithms are also given. The criterion function includes the $L_l,~l\ge1$, the Huber function, the Log-cosh function, and the Quantile function as special cases. The problems that are interesting and deserved for future research include the selection of the specific convex function in the criterion for the robustness of identification, the convergence rate and the asymptotic distribution of the proposed algorithms, the recursive version of the RegLS criterion, etc.%regularized general identification criterion, etc.%in Algorithm \ref{Algm2}

\appendices
\renewcommand\thelemma{A\arabic{lemma}}	
\renewcommand\thedefinition{A\arabic{definition}}	
\def\thesectiondis{\thesection}          %\def\thesubsectiondis{\thesection.\arabic{subsection}.}          \def\thesubsubsectiondis{\thesubsection.\arabic{subsubsection}.}
\renewcommand{\thesection}{\Alph{section}}

\section{General Convergence Theorem of SAAWET and Technical Lemmas}	

Consider the case where the root $x^{0}$ of $f(\cdot): \mathbb{R}^{m} \rightarrow \mathbb{R}^{m}$ is a singleton, i.e., $f(x)=0$ if and only if $x=x^{0}$. Let $\left\{M_{k}\right\}_{k \geq 1}$ be a sequence of positive numbers increasingly diverging to infinity and let the initial value $x_{0}$ be given. We have the noisy observations of $f(\cdot): y_{k+1}=f\left(x_{k}\right)+\varepsilon_{k+1}$, where $\left\{x_{k}\right\}_{k \geq 1}$ is given by the following SAAWET algorithm,
\begin{align}
x_{k+1}=&\left[x_{k}+a_{k} y_{k+1}\right] \mathbb{I}_{\left[\left\|x_{k}+a_{k} y_{k+1}\right\| \leq M_{\sigma_{k}}\right]}\nonumber\\
&+x^{*}\mathbb{I}_{\left[\left\|x_{k}+a_{k} y_{k+1}\right\|>M_{\sigma_{k}}\right]},\\
\sigma_{k}=&\sum_{i=1}^{k-1} \mathbb{I}_{\left[\left\|x_{i}+a_{i} y_{i+1}\right\|>M_{\sigma_{i}}\right]},~\sigma_{0}=0.
\end{align}

For the General Convergence Theorem (GCT) of SAAWET, the following conditions are needed.

\begin{sloppypar}
\begin{itemize}
\item[C1)] $a_{k}>0, a_{k} \mathop{\longrightarrow}\limits_{k \to \infty} 0$, and $\sum_{k=1}^{\infty} a_{k}=\infty$,
\item[C2)] There is a continuous differentiable function $v(\cdot): \mathbb{R}^{m} \rightarrow\mathbb{R}$ such that $\sup _{\delta \leq\left\|x-x^{0}\right\| \leq \Delta} \nabla v(x)^{T} f(x)<0,~\forall~\Delta>\delta>0$. Furthermore, $x^{*}$  is such that $v\left(x^{*}\right)<$ $\inf _{\|x\|=c_{0}} v(x)$ for some $c_{0}>0$ and $\left\|x^{*}\right\|<c_{0}$.
\item[C3)] For a fixed realization $\omega\in\Omega$,
\begin{align}
\lim _{T \rightarrow 0} \limsup _{k \rightarrow \infty} \frac{1}{T}\Big\|\!\!\!\!\sum_{i=n_{k}}^{m\left(n_{k}, T_{k}\right)}\!\!\!\! a_{i} \varepsilon_{i+1}\Big\|=0,~~ \forall T_{k} \in[0, T]
\end{align}	
for any $\left\{n_{k}\right\}_{k \geq 1}$ such that $\left\{x_{n_{k}}\right\}_{k \geq 1}$ converges, where
$
m(k, T)=\max \left\{m: \sum_{j=k}^{m} a_{j} \leq T\right\}.$
\item[C4)] $f(\cdot)$ is measurable and locally bounded.
\end{itemize}
\end{sloppypar}

\setcounter{theorem}{0}
\renewcommand\thetheorem{A\arabic{theorem}}

\begin{theorem}\label{ThmA1} (\cite{chen2014recursive})
Assume that C1), C2), and C4) hold. Then $x_{k} \mathop{\longrightarrow}\limits_{k \rightarrow \infty} x^{0}$ for those $\omega$ such that C3) holds.
\end{theorem}

%\begin{lemma}[Kronecker]\label{Kronecher}
%If $a_{k},b_{k}$ are sequences of real numbers with $0<b_{k}\uparrow \infty$,
%$\sum_{k=1}^{\infty } a_{k}$ converging, then
%$$\frac{1}{b_{N} } \sum_{k=1}^{N} a_{k}b_{k}\underset{N\to\infty}{\longrightarrow}0. $$
%\end{lemma}	

\begin{definition}(\cite{davydov1973mixing})
Suppose $\{z_k\}_{k \geq 1}$ is a random process on a given probability space $\left( \Omega,\mathcal{F},\mathbb{P}\right) $. Let $\mathscr{F}_{0}^{n} \triangleq \sigma\left\{z_{k}, 0 \leq k \leq n\right\}$ and  $\mathscr{F}_{n}^{\infty} \triangleq   \sigma\left\{z_{k}, k \geq n\right\}$  be the $\sigma$-algebras generated by $\left\{z_{k}, 0 \leq k \leq n\right\} $ and  $\left\{z_{k}, k \geq n\right\}$, respectively. The $\phi$-mixing coefficients of the process are defined as follows:		
\begin{align*}
\phi_k \!\triangleq\! \sup _{n} \!\!\sup _{A \in \mathscr{F}_{0}^{n}, \mathbb{P}(A)>0, B \in \mathscr{F}_{n+k}^{\infty}}\!\!\!\!\frac{|\mathbb{P}(A B)-\mathbb{P}(A) \mathbb{P}(B)|}{\mathbb{P}(A)},~~k\geq1.
\end{align*}
The process is said to be $\phi$-mixing if $\lim _{k \rightarrow \infty} \phi_k=0$.
\end{definition}

\begin{lemma}\label{stoas}(\cite{masry1987strong})
Assume that $\left\{x_{k}\right\}_{k \geq 0}$ with $ x_{k} \in \mathbb{R}^{m}$ is $\phi$-mixing with mixing coefficients $\{\phi_k\}_{k \geq 0}$. Let $\left\{\Phi_{k}(\cdot)\right\}_{k \geq 0}$ be a sequence of functions satisfying $\Phi_{k}(\cdot): \mathbb{R}^{m} \rightarrow \mathbb{R}$ and $ \mathbb{E} \Phi_{k}\left(x_{k}\right)=0$. If there exist some constants $\epsilon>0$ and $\gamma>0$ such that
$$
\sum_{k=1}^{\infty}\left(\mathbb{E}\left|\Phi_{k}\left(x_{k}\right)\right|^{2+\epsilon}\right)^{\frac{2}{2+\epsilon}}<\infty
$$
and
$$
\sum_{k=1}^{\infty} \log k(\log \log k)^{1+\gamma}(\phi_k)^{\frac{\epsilon}{2+\epsilon}}<\infty
$$
then
$$
\sum_{k=1}^{\infty} \Phi_{k}\left(x_{k}\right)<\infty \text { a.s. }
$$
\end{lemma}	

\begin{lemma}\label{Lj}(\cite{lennart1999system}, Theorem 8.2)
Assume that
\begin{itemize}
\item[(i)] $R(\theta)$ is a deterministic function that is continuous in $\theta \in \mathcal{M}$ and minimized at the set
$$\mathscr{D}=\left\{\theta \mid \theta \in \mathcal{M}, R(\theta)=\min _{\theta^{\prime} \in \mathcal{M}} R\left(\theta^{\prime} \right)\right\} $$
where $\mathcal{M}$ is a compact subset of $\mathbb{R}^{d}$.
\item[(ii)] A sequence of functions {\color{black}$\{R_{N}(\theta)\}_{N\geq1}$} converges to {\color{black}$R(\theta)$} almost surely and uniformly in $\mathcal{M} $ as $N$ goes to $+\infty$.
\end{itemize}
Then $\widehat{\theta}_{N}=\arg \min _{\theta \in \mathcal{M} } R_{N}(\theta)$ converges to $\mathscr{D}$ almost surely, namely, $\inf _{\theta^{*} \in \mathscr{D}}\left\|\widehat{\theta}_{N}-\theta^{*}\right\| \rightarrow 0$, as $N \to +\infty$.
\end{lemma}	

%\begin{definition}
%Let $(X,d)$ be a metric space and let $A \subseteq X$. We say that $A$ is compact if for every open cover $\left\{U_{\lambda}\right\}_{\lambda \in \Lambda}$ there is a finite collection $U_{\lambda}, \ldots, U_{\lambda_{k}}$ so that $A \subseteq \bigcup_{i=1}^{k} U_{\lambda_{i}}$. In other words a set is compact if and only if every open cover has a finite subcover.
%\end{definition}

For a univariate nonnegative convex function, it is direct to obtain the following result.

\begin{lemma}\label{conv<}
If $\Phi(\cdot)$ is nonnegative and convex on $\mathbb{R}$ with $\Phi(x)>\Phi(0)$, $\forall x \ne0$, then $\Phi(\frac{x+y}{2})< \frac{1}{2}\Phi(x)+\frac{1}{2}\Phi(y)$ for any $x<0<y$.
\end{lemma}

%\begin{IEEEproof}
%Without loss of generality, assume that $\Phi(0)=0$. Thus if $|x|<|y|$, by the convexity of $\Phi(\cdot)$ we have
%\begin{align}
%\Phi(\frac{x+y}{2})\le \frac{\Phi(y)}{y}\cdot\frac{x+y}{2}.
%\end{align}
%In fact, we have $\frac{\Phi(y)}{y}\cdot\frac{x+y}{2}<\frac{\Phi(x)+\Phi(y)}{2}$ since $x<0$.
%In a similar manner we can prove that $\Phi(\frac{x+y}{2})< \frac{1}{2}\Phi(x)+\frac{1}{2}\Phi(y)$, if $x<0<y$ and $|x|>|y|$. Thus we verify that $\Phi(\frac{x+y}{2})< \frac{1}{2}\Phi(x)+\frac{1}{2}\Phi(y)$, if $x<0<y$, where the inequality strictly holds.\\
%\end{IEEEproof}

%\begin{lemma}[$c_l$ inequality]
%$$\left(\sum_{i=1}^{n}\left|a_{i}\right|\right)^{l} \le c_l \sum_{i=1}^{n}\left|a_{i}\right|^{l},$$
%where $c_l=1$ if $0<r\le1$ and $c_l=n^{r-1}$ if $r>1$.
%\end{lemma}
%\begin{lemma}\label{EX-EX}
%If $X$ is a random variable, then for $\forall \gamma>0$, it holds that
%$$\mathbb{E}|X-\mathbb{E} X|^{2+\gamma} \leqslant 2^{2+\gamma} \mathbb{E}|X|^{2+\gamma} .$$
%\end{lemma}			

\section{Proof of Theorem \ref{con}}

First, we prove that if $\Phi(\cdot)$ is convex, then under the given assumptions the function $R(\theta)=\mathbb{E}(\Phi(y_{k+1} -\theta^{T}x_{k}))$ is also convex. In fact, for any fixed $0 \le \lambda \le 1$ and $\theta_{1},\theta_{2} \in \mathbb{R}^{d}$, we have the following inequalities
\begin{align}\label{convex12}
& R\left(\lambda \theta_{1}+(1-\lambda) \theta_{2}\right) \nonumber \\
=& \mathbb{E}\left(\Phi\left(y_{k+1}-\left(\lambda \theta_{1}+(1-\lambda) \theta_{2}\right)^{T} x_{k}\right)\right) \nonumber \\
\le& \mathbb{E}\left(\lambda \Phi\left(y_{k+1}-\theta_{1}^{T} x_{k}\right)+(1-\lambda) \Phi\left(y_{k+1}-\theta_{2}^{T} x_{k}\right)\right) \nonumber\\
=& \lambda R\left(\theta_{1}\right)+(1-\lambda) R\left(\theta_{2}\right),
\end{align}
which indicates that $R(\theta)$ is convex.

Next we prove that $R(\theta)\ge R(\theta^*)$. Using the smoothing property of the conditional expectation (see, e.g., \cite{chow1978probability}) and noting that $w_{k+1}$ is independent of $x_k$ and $x_k$ is with a pdf $q(\cdot)$, by A3) we have 
\begin{align} \label{Rtheta}
R(\theta) &=\mathbb{E}\left[\Phi\left(y_{k+1}-\theta^{T} x_{k}\right)\right] \nonumber \\
%&=\mathbb{E}\left[\mathbb{E}\left[\Phi\left(w_{k+1}-\left(\theta-\theta^{*}\right)^{T} x_{k}\right) \big| x_{k}\right]\right] \nonumber \\
&=\mathbb{E}\left[\mathbb{E}\left[\Phi\left(w_{k+1}-\left(\theta-\theta^{*}\right)^{T} x\right) \right]\Big|_{x=x_{k}}\right] \nonumber \\
%&=\mathbb{E}\left [ {\int_{\mathbb{R}} \Phi(s-\left(\theta-\theta^{*}\right)^{T} x) f_{w}(s) \mathrm{d} s}\Big|_{x=x_{k}} \right ]\nonumber \\
&=\int_{\mathbb{R}^{d}} \int_{\mathbb{R}} \Phi\left(s-\left(\theta-\theta^{*}\right)^{T} u\right) f_{w}(s) q(u) \mathrm{d}s\mathrm{d}u.
\end{align}

Thus we obtain
\begin{align} \label{17}
&R(\theta)-R(\theta^*)\nonumber \\
=&\int_{\mathbb{R}^{d}} \int_{\mathbb{R}} \left[\Phi\left(s-\left(\theta-\theta^{*}\right)^{T} u\right)-\Phi\left(s\right) \right]f_{w}(s) q(u) \mathrm{d}s\mathrm{d}u\nonumber \\
\geq & \int_{\mathbb{R}^{d}} \int_{\mathbb{R}} \partial \Phi\left(s\right)\cdot\left[ -\left(\theta\!\!-\!\!\theta^{*}\right)^{T} u\right]f_{w}(s) q(u) \mathrm{d}s\mathrm{d}u=0,
\end{align}
where $\partial \Phi\left(s\right)$ denotes the sub-gradient of $\Phi\left(s\right)$ and the last equality holds because of $\mathbb{E}x_k=0$.

From (\ref{17}) we know that $\theta^*$ is a minimizer of $R(\cdot)$.

We proceed to prove that $\theta^*$ is the unique minimizer of $R(\cdot)$. Suppose that there exists another minimizer $\hat{\theta}\neq \theta^*$ such that $R(\hat{\theta})=R(\theta^*)$. By the convexity of $R(\theta)$, it follows that for any fixed $\lambda \in (0,1)$, $\lambda \hat{\theta}+ (1-\lambda)\theta^*$ is also a minimizer of $R(\theta)$ and
\begin{align}\label{zero}
R(\lambda \hat{\theta}+ (1-\lambda)\theta^*)-\lambda R(\hat{\theta})-(1-\lambda)R(\theta^*)=0.
\end{align}

By the definition of $R(\theta)$ and the convexity of $\Phi(\cdot)$, for any fixed $\lambda\in(0,1)$ we have
\begin{align}
&R(\lambda \hat{\theta}+ (1-\lambda)\theta^*)-\lambda R(\hat{\theta})-(1-\lambda)R(\theta^*)\nonumber\\
=&\mathbb{E}(\Phi(y_{k+1} -(\lambda \hat{\theta}+ (1-\lambda)\theta^*)^{T}x_{k}))-\mathbb{E}(\lambda \Phi(y_{k+1} - \hat{\theta}^{T}x_{k}))\nonumber\\
&-\mathbb{E}((1-\lambda) \Phi(y_{k+1} - \theta^{*T}x_{k}))\nonumber\\
=&\int_{\mathbb{R}^{d}} \int_{\mathbb{R}}\Big[ \Phi(s-\lambda u^{T}(\hat{\theta}-\theta^*))- \lambda \Phi(s-u^{T}(\hat{\theta}-\theta^{*}))\nonumber\\
& -(1-\lambda)\Phi(s)\Big]\cdot f_{w}(s)q(u)\mathrm{d}s\mathrm{d}u\le0,
\end{align}
where the last inequality holds true because
\begin{align}
\nonumber &\Phi(s-\lambda u^{T}(\hat{\theta}-\theta^*))- \lambda \Phi(s-u^{T}(\hat{\theta}-\theta^{*}))-(1-\lambda) \Phi(s)\\
=&\Phi(\lambda(s- u^{T}(\hat{\theta}-\theta^*))+(1-\lambda)s)\nonumber\\
&-\lambda \Phi(s-u^{T}(\hat{\theta}-\theta^{*}))-(1-\lambda) \Phi(s)
\le 0.\label{24}
\end{align}

Let $g(\lambda,s,u)\triangleq\Phi(s-\lambda u^{T}(\hat{\theta}-\theta^*))- \lambda \Phi(s-u^{T}(\hat{\theta}-\theta^{*}))-(1-\lambda) \Phi(s)$.
In the following we are going to prove that there exists some $\lambda\in(0,1)$ and a set $\mathcal{D}\subset \mathbb{R}^d\times\mathbb{R} $, whose Lebesgue measure is greater than 0, such that
\begin{align}\label{contraction}
 \int_{\mathcal{D}}g(\lambda,s,u)f_{w}(s)q(u)\mathrm{d}s\mathrm{d}u<0.
\end{align}

By Lemma \ref{conv<} we notice that $\Phi(\frac{x+y}{2})< \frac{1}{2}\Phi(x)+\frac{1}{2}\Phi(y)$ for any $x<0<y$.
Set $\mathcal{D}=\{(u ,s)\in \mathbb{R}^d\times\mathbb{R}~|~0<s<u^{T}(\hat{\theta}-\theta^{*})\}$ and thus $g(\frac{1}{2},s,u)<0$ for $(s,u)\in\mathcal{D}$. By noting that $f_{w}(\cdot)$ is positive and continuous at the origin, there exists $\delta_1>0$ such that $f_w(s)>0$ if $0<s<\delta_1$. Noting also the assumption A2), we can find a set $\mathcal{A}\subset\mathbb{R}^d$ with Lebesgue measure being greater than $0$ such that $q(u)>0$ for $u\in\mathcal{A}$.

Noting that the Lebesgue measure of the set $\{u\in\mathbb{R}^d:0< u^{T}(\hat{\theta}-\theta^{*})\}\cap B_r(0) $ is greater than $0$, there exists $\delta_2>0$ such that the Lebesgue measure of the set $\mathcal{A}\triangleq\{u\in\mathbb{R}^d~|~\delta_2< u^{T}(\hat{\theta}-\theta^{*})\}\cap B_r(0) $ is also greater than $0$. Without losing generality, we assume that $\delta_1<\delta_2$.

Thus we have $\mathcal{A}\times(0,\delta_1)\subset \mathcal{D}$ and $g(\frac{1}{2},s,u)f_w(s)q(u)<0$ for $(u,s)\in \mathcal{A} \times (0,\delta_1)$, which leads to
\begin{align*}
&\int_{\mathcal{D}}g\left(\frac{1}{2},s,u\right)f_{w}(s)q(u)\mathrm{d}s\mathrm{d}u\\
&\le \int_{ \mathcal{A}\times (0,\delta_1) }g\left(\frac{1}{2},s,u\right)f_{w}(s)q(u)\mathrm{d}s\mathrm{d}u<0    
\end{align*}
and (\ref{contraction}) is proved for $\lambda=\frac{1}{2}$, which contradicts with (\ref{zero}) and (\ref{24}). Then the conclusion that $\theta^{*}$ is the unique minimizer of $R(\theta)$ follows. This finishes the proof.

\section{Proof of Lemma \ref{Nlim0}}

Here we only prove (\ref{e1N}), while (\ref{e2N}) and (\ref{xN}) can be treated similarly. By the Kronecker lemma, for (\ref{e1N}) we only need to prove
\begin{align}\label{e1}
\sum_{k=1}^{\infty}\frac{1}{k}\left(\|x_k\||y_{k+1}|^l-\mathbb{E}\|x_k\||y_{k+1}|^l \right)< \infty \text{ a.s.}
\end{align}
with $l>0$ specified in assumption A1).

Noting that $\{x_k\}_{k\ge0}$ is $\phi$-mixing with geometric mixing coefficients, by Lemma \ref{stoas} given in the Appendix, for (\ref{e1}) it suffices to prove that for some $\epsilon>0$,
\begin{equation}\label{111}
\sum_{k=1}^{\infty} \frac{1}{k^2}\left( \mathbb{E}\left|\left\| x_{k}\right\|\left|y_{k+1}\right|^{l}-\mathbb{E} \left\| x_{k}\right\|\left|y_{k+1}\right|^{l}\right|^{2+\epsilon}\right) ^{\frac{2}{2+\epsilon}}<\infty.
\end{equation}

Noting the mutual independence of $x_k$ and $w_{k+1}$, direct calculation leads to
\begin{align}\label{23}
&\mathbb{E} \left( \left\| x_{k}\right\|\left|y_{k+1}\right|^{l}\right) ^{2+\epsilon}\nonumber\\
\le& \mathbb{E} \left( \left\|x_{k}\right\|\left|w_{k+1}\right|^l+\left\|\theta^*\right\|^l\left\|x_k\right\|^{l+1} \right)^{2+\epsilon}\nonumber\\
\le& c\Big( \mathbb{E}\left\|x_{k}\right\|^{2+\epsilon}\mathbb{E}\left|w_{k+1}\right|^{l(2+\epsilon})+\mathbb{E}\left\|x_{k}\right\|^{(l+1)(2+\epsilon)}\Big),
\end{align}
where $c>0$ is a constant.
By assumption A2) and A3), we have that $\mathbb{E}\left\|x_{k}\right\|^{(l+1)(2+\epsilon)}<\infty$ and $\mathbb{E}\left|w_{k+1}\right|^{l(2+\epsilon})<\infty$ for sufficiently small $\epsilon>0$, which combining with (\ref{23}) leads to $\mathbb{E} \left( \left\| x_{k}\right\|\left|y_{k+1}\right|^{l}\right) ^{2+\epsilon}<\infty$ and accordingly (\ref{111}) and (\ref{e1}) holds. Thus (\ref{e1N}) is proved. This finishes the proof.

\section{Proof of Theorem \ref{The2}}

It follows from (\ref{RN}) that
\begin{align}\label{9898}
&R_{N}(\theta)-R(\theta)\nonumber\\
&=\frac{1}{N}\sum_{k=1}^{N}\Big(\Phi(y_{k+1} -\theta ^{T}x_{k})-\mathbb{E}(\Phi(y_{k+1} -\theta ^{T}x_{k}))\Big).
\end{align}

To show $R_{N}(\theta)-R(\theta)\underset{N \rightarrow \infty}{\longrightarrow}0$ a.s., by the Kronecker lemma (See Appendix A), it suffices to prove 					
\begin{align}\label{KR}\!\!\!\!
\sum_{k=1}^{\infty } \frac{1}{k}\Big(\Phi(y_{k+1} \!-\!\theta ^{T}x_{k})\!-\!\mathbb{E}(\Phi(y_{k+1} \!\!-\!\!\theta ^{T}x_{k}))\Big)\!<\!\infty \text { a.s. }
\end{align}
By Lemma \ref{stoas} given in Appendix, for (\ref{KR}) we only need to show
\begin{align}\label{mixing}
\sum_{k=1}^{\infty }& \frac{1}{k^{2}}\Big(\mathbb{E}\big|\Phi \left (y_{k+1} -\theta ^{T}x_{k}\right) \nonumber\\
&-\mathbb{E} \left(\Phi \left (y_{k+1} -\theta ^{T}x_{k}\right)\right)\big|^{2+\epsilon}\Big)^{\frac{2}{2+\epsilon}} <\infty,
\end{align}
for some $\epsilon>0$.

We first prove that there exists $\epsilon>0$ such that $\sup_k\mathbb{E}\left |\Phi(y_{k+1} -\theta ^{T}x_{k}) \right |^{2+\epsilon}<\infty$. By the assumption on $\Phi(\cdot)$ and carrying out a similar analysis as (\ref{Rtheta}), for any fixed $\epsilon>0$, we have
\begin{align}\label{Ephi}
&\mathbb{E}\left |\Phi(y_{k+1} -\theta ^{T}x_{k})  \right |^{2+\epsilon}\nonumber\\
=&\int_{\mathbb{R}^{d}} \int_{\mathbb{R}}\left| \Phi\left(s-\left(\theta-\theta^{*}\right)^{T} u\right) \right|^{2+\epsilon}f_{w}(s) q(u) \mathrm{d}s\mathrm{d}u \nonumber\\
%\le& \int_{\mathbb{R}^{d}} \int_{\mathbb{R}}\left[ c\left|s-\left(\theta-\theta^{*}\right)^{T} u\right|^{l}+1 \right]^{2+\epsilon}f_{w}(s) q(u) \mathrm{d}s\mathrm{d}u \nonumber\\
\le& c\int_{\mathbb{R}^{d}}\!\! \int_{\mathbb{R}}\!\!\left[ \left|s\right|^{l\left(2+\epsilon\right)}\!\!+\!\! \left\|\theta\!-\!\theta^{*}\right\|^{l\left(2+\epsilon\right)}\!\left| u\right|^{l\left(2+\epsilon\right)}\!+\!1 \right]\!\!f_{w}(s) q(u) \mathrm{d}s\mathrm{d}u 
\end{align}	
where $c>0$ is a constant which may vary among different lines of inequalities.

By assumption A2) and A3), for some $\epsilon_0>0$, {\color{black}$\mathbb{E}\left\|x_{k}\right\|^{2(l+2)+\epsilon}_0<\infty$} and $\mathbb{E}\left|w_{k+1}\right|^{2l+\epsilon}_0<\infty$ for sufficiently small $\epsilon>0$, and hence
\begin{align}\label{69}
\sup_k\mathbb{E}\left |\Phi(y_{k+1} -\theta ^{T}x_{k})  \right |^{2+\epsilon}< \infty
\end{align}
for any $0 \le \epsilon \le \frac{\epsilon_{0}}{l}$.

Noting that there exists $c_\epsilon>0$ such that $\mathbb{E}\left|\Phi \left (y_{k+1} -\theta ^{T}x_{k}\right) -\mathbb{E} \left(\Phi \left (y_{k+1} -\theta ^{T}x_{k}\right)\right)\right|^{2+{\epsilon}}\le c_\epsilon\mathbb{E}\left |\Phi(y_{k+1} -\theta ^{T}x_{k})  \right |^{2+{\epsilon}}$, the validity of \eqref{mixing} is verified according to \eqref{69}. Hence (\ref{KR}) is proved and (\ref{RNR}) holds true.

Next we prove (\ref{minlim}). The key steps are as follows. We first show that with respect to the variable $\theta$, the function sequence $\left\{R_{N}(\theta)\right\}_{N\geq1}$ converges to $R(\theta)$ uniformly in $\mathcal{M}\subset\mathbb{R}^{d}$, where $\mathcal{M}$ can be any fixed compact set in $\mathbb{R}^{d}$ containing $\theta^{*}$ as an interior point. Then we prove that
\begin{align*}
\underset{\theta \in \mathcal{M}}{\operatorname{argmin}} R_{N}(\theta)=\underset{\theta \in  \mathbb{R} ^{d}}{\operatorname{argmin}} R_{N}(\theta),
\end{align*}
for all $N$ large enough and obtain (\ref{minlim}) by Lemma \ref{Lj} given in the Appendix.

{\em Step 1)} We first show that $\left\{R_{N}(\theta)\right\}_{N\geq1}$ converges to $R(\theta)$ uniformly in any fixed compact set $\mathcal{M} \subset \mathbb{R} ^{d}$.

By Lemma \ref{Nlim0} and noting that $\{x_k\}_{k\ge1}$ and $\{w_{k+1}\}_{k\ge1}$ are both identically distributed, respectively, we have
\begin{align}\label{qianti}
&\frac{1}{N}\sum_{k=1}^{N} (\left\|x_{k}\right\|\left|y_{k+1}\right|^{l}+\left\|x_{k}\right\|^{l+1}+\left\|x_{k}\right\|)\nonumber\\
&\overset{N\to\infty}{\longrightarrow} \mathbb{E}\left(\left\|x_{1}\right\|\left|y_{2}\right|^{l}+\left\|x_{1}\right\|^{l+1}+\left\|x_{1}\right\|\right)~~ \text{a.s.}
\end{align}

From (\ref{qianti}) it follows that there exists a set $\Omega_{0}\subset \Omega$ with $\mathbb{P}\left\{\Omega_{0}\right\}=1 $ such that (\ref{qianti}) holds along any fixed sample path $\omega\in \Omega_0$. In the following analysis we will focus on a fixed sample path in $\Omega_0$.

By (\ref{qianti}), for some $c>0$ we have
\begin{align}\label{bound}
\frac{1}{N}\sum_{k=1}^{N} \left(\left\|x_{k}\right\|\left|y_{k+1}\right|^{l}+\left\|x_{k}\right\|^{l+1}+\left\|x_{k}\right\|\right) \le c, ~~\forall~N\geq1.
\end{align}

For any fixed $\theta_{1}, \theta_{2} \in \mathcal{M}$, we have
\begin{align}
&|R_{N}(\theta_{1})-R_{N}(\theta_{2})| \nonumber\\
%&=\left | \frac{1}{N}\sum_{k=1}^{N}\Phi(y_{k+1} -\theta_{1} ^{T}x_{k})-\frac{1}{N}\sum_{k=1}^{N}\Phi(y_{k+1} -\theta_{2} ^{T}x_{k}) \right |  \nonumber\\
&\le \frac{1}{N}\sum_{k=1}^{N}\left | \Phi(y_{k+1} -\theta_{1} ^{T}x_{k})-\Phi(y_{k+1} -\theta_{2} ^{T}x_{k}) \right |.
\end{align}

Noting Remark \ref{Remk3}, we have	
\begin{align}\label{21}
&|R_{N}(\theta_{1})-R_{N}(\theta_{2})| \nonumber \\
\le& \frac{1}{N}\!\!\sum_{k=1}^{N}\!|\partial\Phi(y_{k+1}-\theta_{1} ^{T}x_{k})\!+\!\partial\Phi(y_{k+1}\!-\!\theta_{2} ^{T}x_{k})|\!\!\left\|x_{k}\right\|\!\!\left\|\theta_{1}\!-\!\theta_{2}\right\|.
\end{align}

Since the set $\mathcal{M}$ is compact, without loss of generality, we assume $\|\theta\| \le M$, $\forall~\theta \in \mathcal{M}$.
By Remark \ref{Remk3}, we have that
\begin{align}\label{22}
&|\partial\Phi(y_{k+1}-\theta_{1} ^{T}x_{k})+\partial\Phi(y_{k+1}-\theta_{2} ^{T}x_{k})|\nonumber\\
\le& c\left(|y_{k+1}-\theta_{1} ^{T}x_{k}|^l+|y_{k+1}-\theta_{2} ^{T}x_{k}|^l+1\right)\nonumber\\
\le& c\left(\left|y_{k+1}\right|^{l}\!+\!M\left\|x_{k}\right\|^{l}\!+\!1\right)\le c \left(\left|y_{k+1}\right|^{l}\!+\!\left\|x_{k}\right\|^{l}\!+\!1\right),
\end{align}
where $c>0$ is a constant which may vary among different lines of inequalities.

Combining (\ref{bound}), (\ref{21}), and (\ref{22}), we have that
\begin{align}\label{equic}
&|R_{N}(\theta_{1})-R_{N}(\theta_{2})| \nonumber\\
%\le& \frac{1}{N}\sum_{k=1}^{N}|\partial\Phi(y_{k+1}-\theta_{1} ^{T}x_{k})+\partial\Phi(y_{k+1}-\theta_{2} ^{T}x_{k})|\left\|x_{k}\right\|\left\|\theta_{1}-\theta_{2}\right\| \nonumber\\
\le& c\frac{1}{N}\sum_{k=1}^{N} \left(\left\|x_{k}\right\|\left|y_{k+1}\right|^{l}+\left\|x_{k}\right\|^{l+1}+\left\|x_{k}\right\|\right) \left\|\theta_{1}-\theta_{2}\right\| \nonumber\\
\le& c\left\|\theta_{1}-\theta_{2}\right\|, ~~\forall~N\geq1,
\end{align}
where the constant $c>0$ does not depend on $N$ varying among different lines of inequalities.

In fact, (\ref{equic}) shows that the function sequence $\left\{R_{N}(\theta)\right\}_{N\ge1}$ is equi-continuous on $\mathcal{M}$, that is, for any $\varepsilon>0$, there exists $\delta=\frac{\varepsilon}{3c}>0$ with $c>0$ specified in (\ref{equic}) such that
\begin{align}\label{delta}
|R_{N}(\theta_{1})-R_{N}(\theta_{2})|<\frac{\varepsilon}{3},~~\forall~N\geq1,
\end{align}
for any $\theta_1, \theta_2 \in \mathcal{M}$, $\left\|\theta_{1}-\theta_{2}\right\|<\delta$.

With $\delta>0$ mentioned above, it directly follows by the compactness of $\mathcal{M}$ that
\begin{align}\label{41}
\mathcal{M}\subseteq \bigcup_{s=1}^{q} B\left(\theta_{s}, \delta\right),
\end{align}
where $B\left(\theta, \delta\right)=\{x\in \mathbb{R}^d~|~\|x-\theta\|<\delta\}$ is a ball with the radius $\delta$ centered at $\theta$.

Because $\{R_{N}(\theta)\}_{N\geq1}$ converges to $R(\theta)$ at $\theta=\theta_{1}, \cdots, \theta_{q}$, for $\varepsilon>0$ mentioned above, there exists $\widetilde{N} $ large enough such that
\begin{align}\label{convergeq}
|R_{n}(\theta_{s})-R_{m}(\theta_{s})|<\frac{\varepsilon}{3},~~ s=1,2,\cdots,q,  ~~\forall~ n,m >\widetilde{N}.
\end{align}

Noting (\ref{41}), for any $\theta \in \mathcal{M}$, there exists some $\theta_{s}$, $s\in \left \{ 1,2,\cdots ,q \right \}$ such that $\theta \in B\left(\theta_{s}, \delta\right)$. By (\ref{delta}), we have
\begin{align}\label{embrass}
\left|R_{N}(\theta)-R_{N}\left(\theta_{s}\right)\right|<\frac{\varepsilon}{3},~~\forall~N\geq1.
\end{align}

Then combining (\ref{convergeq}) and (\ref{embrass}), we have that for $\forall~n,m >\widetilde{N} $ and any $\theta \in \mathcal{M}$, there exists $\theta_{s}$ such that
\begin{align}\label{uniform}
&|R_{n}(\theta)-R_{m}(\theta)|\nonumber\\
\leq&|R_{n}(\theta)\!-\!R_{n}(\theta_{s})|\!\!+\!\!|R_{n}(\theta_{s})\!-\!R_{m}(\theta_{s})|\!\!+\!\!|R_{m}(\theta_{s})\!-\!R_{m}(\theta)| \!<\! \varepsilon.
\end{align}
It follows from (\ref{uniform}) and (\ref{RNR}) that as $N\longrightarrow \infty $, $\{R_{N}(\theta)\}_{N\geq1}$ converges to $R(\theta)$ uniformly in $\mathcal{M}$ {\color{black}and almost surely in $\Omega$}.

{\em Step 2)} Next we show that $\bar{\theta}_{N}=\arg \min _{\theta \in \mathcal{M}} R_{N}(\theta)$ converges to $\theta^{*}$.

{\color{black}By far, we obtain the almost sure and uniform convergence of $R_{N}(\theta)$. Noting that $\theta^{*}$ is the unique minimizer of $R(\theta)$} satisfying $\theta^*\in \mathcal{M}$, by Lemma \ref{Lj} given in the Appendix, it directly follows that
\begin{align}\label{compact}
\bar{\theta}_{N} \underset{N \rightarrow \infty}{\longrightarrow} \theta^{*}~~\text { a.s. }
\end{align}		

{\em Step 3)} Finally, we prove that for $\widehat{\theta}_{N}$ generated from (\ref{7}), $\bar{\theta}_{N}=\widehat{\theta}_{N}$ for all $N$ large enough.

Noting that $\theta^{*}$ is an interior point of $\mathcal{M}$ and also (\ref{compact}), we can obtain that $\bar{\theta}_{N},N\geq1$ are also interior points of $\mathcal{M}$ for all $N$ large enough. By the convexity of $R_N(\theta)$, this implies that
\begin{align}
\bar{\theta}_{N}=\arg \min _{\theta \in \mathcal{M}} R_{N}(\theta)=\arg \min _{\theta \in \mathbb{R}^d} R_{N}(\theta)\label{46}
\end{align}
and thus $\bar{\theta}_{N}=\widehat{\theta}_{N}$, for all $N$ large enough. Combining (\ref{compact}) and (\ref{46}), the conclusion (\ref{minlim}) is proved.

\section{Proof of Proposition \ref{Prop1}}

By Theorem \ref{con}, we know that $\theta^*$ is the unique minimizer of the convex function $R(\theta)$. Thus it suffices to prove that $\nabla R(\theta)=-\mathbb{E}\left[\varphi\left(y_{k+1}-\theta^{T} x_{k}\right) \cdot x_{k}\right]$.

Set $\theta=[\theta_1,\cdots,\theta_d]^T\in\mathbb{R}^d$. We consider $\frac{\partial R(\theta)}{\partial \theta _{i}}$, where $\theta_{i}$ is the $i$-th component of $\theta$, and it follows that
\begin{align}\label{Lebe}
\frac{\partial R(\theta)}{\partial \theta _{i}}
=&\lim_{h \to 0}\int_{\mathbb{R}^{d}} \int_{\mathbb{R}} \Big[\frac{\Phi\left(s-(\theta+he_i-\theta^*)^Tu\right)}{h}\nonumber\\
&-\frac{\Phi(s-\left(\theta-\theta^{*}\right)^{T} u)}{h}\Big]f_{w}(s) q(u) \mathrm{d}s\mathrm{d}u,
\end{align}
where $e_{i}$ being the vector in $\mathbb{R}^d$ with the $i$-th component being $1$ and others being zero.

Noting that $\Phi(\cdot)$ has a derivative $\varphi(\cdot)$, by the mean value theorem, there exists $|\bar{h}|\le|h|\le1$ such that
\begin{align}
&\Big|\frac{\Phi\left(s-(\theta+he_i-\theta^*)^Tu\right) -\Phi(s-(\theta-\theta^{*})^{T} u)}{h}\Big|\nonumber
\end{align}
\begin{align}
=&\left|\varphi\left ( s-(\theta-\theta^{*})^{T} u+\bar{h}u_{i}  \right ) u_{i}\right|\nonumber\\ 
%\le&c \|u\| \left ( \left | s-(\theta-\theta^{*})^{T} u+\bar{h}u_{i} \right |^{l}+1  \right )\nonumber\\
\le& c \|u\|\left[\left|s\right|^l+\|\theta-\theta^*\|^l\|u\|^l+\|u\|^l+1\right],
\end{align}
where $u_{i}$ is the $i$-th component of $u \in \mathbb{R}^{d}$ and $c>0$ is a constant which may vary among different lines of inequalities.

By the assumption A2) and A3), we know that $\mathbb{E}\|x_k\|^{l+1}<\infty$ and $\mathbb{E}|w_{k+1}|^l<\infty$. Thus
\begin{align}\label{81}
\int_{\mathbb{R}^{d}}\!\!\! \int_{\mathbb{R}} \left ( \left | s \right |^{l}\left \| u \right \| +\!\left \| u \right \|^{l+1}\! +\!\left \| u \right \|  \right )f_{w}(s) q(u) \mathrm{d}s\mathrm{d}u < \infty.
\end{align}

By the Lebesgue dominated convergence theorem{\color{black}\cite{chow1978probability}}, from (\ref{Lebe}) it follows that
\begin{align}
\frac{\partial R(\theta) }{\partial \theta _{i}}
=& \int_{\mathbb{R}^{d}} \int_{\mathbb{R}} \lim_{h \to 0} \Big[\frac{\Phi\left(s-(\theta+he_i-\theta^*)^Tu\right)}{h}\nonumber\\
&-\frac{\Phi(s-\left(\theta-\theta^{*}\right)^{T} u)}{h}\Big] {\color{black}f_{w}(s)} q(u) \mathrm{d}s\mathrm{d}u \nonumber \\
=&-\int_{\mathbb{R}^{d}}\!\! \int_{\mathbb{R}}\!\!\! \varphi ( s\!-\!(\theta\!-\!\theta^{*})^{T} u  )u_{i}f_{w}(s) q(u) \mathrm{d}s\mathrm{d}u.
\end{align}
And thus
\begin{align}\label{dev-int}
\frac{\partial R(\theta)}{\partial\theta}
=&-\!\!\int_{\mathbb{R}^{d}}\!\! \int_{\mathbb{R}}\!\! \varphi\left ( s\!-\!(\theta-\theta^{*})^{T} u \right )  u f_{w}(s) q(u) \mathrm{d}s\mathrm{d}u.
\end{align}

On the other hand, direct calculation leads to
\begin{align}
&\mathbb{E}\left[\varphi\left(y_{k+1}-\theta^{T} x_{k}\right) \cdot x_{k}\right]\nonumber\\
=&\int_{\mathbb{R}^{d}} \int_{\mathbb{R}} \varphi\left ( s-(\theta-\theta^{*})^{T} u \right )  u f_{w}(s) q(u) \mathrm{d}s\mathrm{d}u.\label{59}
\end{align}

Combining (\ref{dev-int}) and (\ref{59}), we prove that $\nabla R(\theta)=-\mathbb{E}\left[\varphi\left(y_{k+1}-\theta^{T} x_{k}\right) \cdot x_{k}\right]$ and this finishes the proof.

\section{Proof of Lemma \ref{Le2}}

Note that there are random truncations in Algorithm \ref{Algm3}. We first consider the case $\lim_{k \to \infty} \sigma _k<\infty $.
Then there exists an $ \alpha>0 $ such that {\color{black}$\|\theta_{k}\|< \alpha,~k \geq 1$} and for some sufficiently large $k_{0}$, there is no truncation for $\theta_{k},~k>n_{k_{0}}$ and thus
\begin{align}\label{notrunc}
\theta_{j+1}=\theta_{n_{k}}+\sum_{i=n_{k}}^{j} a_{i} x_{i} \varphi\left(y_{i+1}-\theta_{i}^{T} x_{i}\right).
\end{align}

It follows from (\ref{notrunc}) and A1') that for $k\geq k_0$ and $ j=n_{k}, \ldots, m\left(n_{k}, T\right)$,		
\begin{align}\label{ineq}
&\left\|\theta_{j+1}-\theta_{n_{k}}\right\|
=\left\|  \sum_{i=n_{k}}^{j} a_{i}  x_{i} \varphi\left(y_{i+1}-\theta_{i}^{T} x_{i}\right)  \right\| \nonumber\\ 
%\le& \sum_{i=n_{k}}^{j} \frac{1}{i}\left\|x_{i}\right\|\left|\varphi\left(y_{i+1}-\theta_{i}^{T} x_{i}\right)\right|\nonumber\\
\le& c\sum_{i=n_{k}}^{j} \frac{1}{i}\left\|x_{i}\right\|\left(\left|y_{i+1}-\theta_{i}^{T} x_{i}\right|^l+1\right) \nonumber\\
\le& c\sum_{i=n_{k}}^{j} \frac{1}{i}\|x_{i}\|\left[\left|y_{i+1}\right|^l+\alpha^l\| x_{i}\|^l+1\right]\nonumber
\end{align}
\begin{align}
%\le& c \sum_{i=n_{k}}^{j} \frac{1}{i}\left\|x_{i}\right\| \left(|y_{i+1}|^{l}+\left\|x_{i}\right\|^{l}+1\right)\nonumber \\
=& c\! \sum_{i=n_{k}}^{j} \frac{1}{i}\!(\left\| x_{i}\right\|\left|y_{i+1}\right|^{l}\!\!-\!\mathbb{E}\left\| x_{i}\right\|\left|y_{i+1}\right|^{l})
+c\!\! \sum_{i=n_{k}}^{j}\!\! \frac{1}{i} \mathbb{E}\left\| x_{i}\right\|\left|y_{i+1}\right|^{l}\nonumber\\
&+c \sum_{i=n_{k}}^{j} \frac{1}{i}(\left\|x_{i}\right\|^{l+1}-\mathbb{E}\left\| x_{i}\right\| ^{l+1})
+ c \sum_{i=n_{k}}^{j}\frac{1}{i} \mathbb{E}\left\| x_{i}\right\| ^{l+1}\nonumber\\
&+c \sum_{i=n_{k}}^{j} \frac{1}{i}(\left\|x_{i}\right\|-\mathbb{E}\left\| x_{i}\right\| )
+c \sum_{i=n_{k}}^{j} \frac{1}{i} \mathbb{E}\left\|x_{i}\right\|,
\end{align}	
where $c>0$ is a constant which may vary among different lines of inequalities.

Similar to the proof of Lemma \ref{con}, we have that as $k \rightarrow \infty$
\begin{align}
&\sum_{i=n_{k}}^{j} \frac{1}{i}(\left\| x_{i}\right\|\left|y_{i+1}\right|^{l}-\mathbb{E}\left\| x_{i}\right\|\left|y_{i+1}\right|^{l})=o(1),\label{xyl}\\
&\sum_{i=n_{k}}^{j} \frac{1}{i}(\left\|x_{i}\right\|^{l+1}-\mathbb{E}\left\| x_{i}\right\| ^{l+1})=o(1)\label{xl+1},
\end{align}
and
\begin{align}
\sum_{i=n_{k}}^{j} \frac{1}{i}(\left\|x_{i}\right\|-\mathbb{E}\left\| x_{i}\right\| )=o(1)\label{x1}.
\end{align}	

By the definition of $m\left(n_{k}, T\right)$ and for $j \in \{n_{k}, \ldots, m\left(n_{k}, T\right)\}$, we have $\sum_{i=n_{k}}^{j} a_{i} \mathbb{E}\left\| x_{i}\right\|\left|y_{i+1}\right|^{l}=O(T)$,  $\sum_{i=n_{k}}^{j} a_{i} \mathbb{E}\left\| x_{i}\right\|^{l+1}=O(T)$, and $\sum_{i=n_{k}}^{j} a_{i} \mathbb{E}\left\| x_{i}\right\|=O(T)$. Then it follows that
\begin{align}
\left\|\theta_{j+1}-\theta_{n_{k}}\right\|=o(1)+O(T),~~j=n_{k}, \ldots, m(n_{k}, T)
\end{align}
and hence (\ref{cT}) holds.

We now consider the case that $\lim_{k \to \infty} \sigma _k=\infty$. Noticing that $\left\{\theta_{n_{k}}\right\}_{k \geq 1}$  is a convergent subsequence of $\left\{\theta_{k}\right\}_{k \geq 1}$, i.e., $\theta_{n_{k}} \rightarrow \bar{\theta}$ as $k \to \infty$, we have that for some $\alpha>0$ and all $k\ge1$, $\left\|\theta_{n_{k}}\right\|\le\alpha$.	

Since $\sigma_{k} \underset{k\to \infty}{\longrightarrow}\infty$, there exists some integer $K>0$ such that		
\begin{align}
M_{\sigma_{n_{k}}}>2 \alpha,~~\forall~k \ge K.
\end{align}

Noting that the truncation bounds $\{M_{k}=k^{\frac{1}{1+2l}}\}_{k\geq1}$, we have
\begin{align}\label{thetabd}
\left\|\theta_{k+1}\right\| \leq M_{\sigma_{k}} \leq M_{k}=k^{\frac{1}{1+2l}}.
\end{align}

We first consider $\theta_{i+1}, i=n_{k}, \ldots, n_{k}+{\color{black}[n_{k}^{\frac{1}{1+2 l}}]}$. For all $k$ sufficiently large, it holds that  $n_{k}+{\color{black}[n_{k}^{\frac{1}{1+2 l}}]}<m\left(n_{k}, T\right)$ because
\begin{align}
\nonumber &\lim_{k \to \infty } \frac{1}{n_{k} } +\frac{1}{n_{k}+1 }+\cdots +\frac{1}{n_{k}+{\color{black}[n_{k}^{\frac{1}{1+2 l}}]} }\\
\le& \lim_{k \to \infty }\int_{n_{k}-1}^{n_{k}+{\color{black}[n_{k}^{\frac{1}{1+2 l}}]}} \!\!\frac{1}{x} \mathrm{d}x=
\lim_{k \to \infty }  \ln \frac{n_{k}+{\color{black}[n_{k}^{\frac{1}{1+2 l}}]}}{n_{k}-1}=0.\label{nk_l}
\end{align}

Similarly to the analysis of (\ref{ineq}) and noting (\ref{thetabd}), we have
\begin{align}\label{truncinf}
&\|\theta_{n_{k}}+\sum_{i=n_{k}}^{n_{k}+[n_{k}^{\frac{1}{1+2 l}}]} a_{i}x_{i} \varphi(y_{i+1}-\theta_{i}^{T}  x_{i}) \| \nonumber\\
\end{align}
\begin{align}
%&\le \left\|\theta_{n_{k}}\right\| + \sum_{i=n_{k}}^{n_{k}+[n_{k}^{\frac{1}{1+2 l}}]} a_{i}\left\|x_{i}\right\|  \left|\varphi(y_{i+1}-\theta_{i}^{T} x_{i})\right| \nonumber \\
\le& \left\|\theta_{n_{k}}\right\| +c \sum_{i=n_{k}}^{n_{k}+[n_{k}^{\frac{1}{1+2 l}}]} a_{i}\left\| x_{i} \right\|({\color{black}\left| y_{i+1}\right|^{l}+\left\| \theta_{i}\right\|^{l}\left\|x_{i}\right\|^{l}}+1) \nonumber \\
\le& \left\|\theta_{n_{k}}\right\| + c \sum_{i=n_{k}}^{n_{k}+[n_{k}^{\frac{1}{1+2 l}}]} \frac{1}{i}\left\| x_{i}\right\|\left|y_{i+1}\right|^{l}\nonumber\\
&+c \sum_{i=n_{k}}^{n_{k}+[n_{k}^{\frac{1}{1+2 l}}]}\frac{1}{i} \cdot i^{\frac{l}{1+2l}}\left\| x_{i}\right\|^{l+1} + c\sum_{i=n_{k}}^{n_{k}+[n_{k}^{\frac{1}{1+2 l}}]} \frac{1}{i}\left\| x_{i}\right\| \nonumber \\
=& \left\|\theta_{n_{k}}\right\|+\!c\!\sum_{i=n_{k}}^{n_{k}+[n_{k}^{\frac{1}{1+2 l}}]}\!\frac{1}{i} (\left\| x_{i}\right\|\left|y_{i+1}\right|^{l}-\mathbb{E}\left\| x_{i}\right\|\left|y_{i+1}\right|^{l})\nonumber\\
&+
c\!\sum_{i=n_{k}}^{n_{k}+[n_{k}^{\frac{1}{1+2l}}]}\!{\frac{1}{i^{\frac{1+l}{1+2l}}}}(\left\| x_{i}\right\|^{l+1}\!-\mathbb{E}\left\|x_{i}\right\|^{l+1})\nonumber\\
&+c\!\!\!\!\!\sum_{i=n_{k}}^{n_{k}+[n_{k}^{\frac{1}{1+2l}}]}\!\!\frac{1}{i}(\left\| x_{i}\right\|-\mathbb{E}\left\|x_{i}\right\|)+c\!\!\!\!\!\sum_{i=n_{k}}^{n_{k}+{\color{black}[n_{k}^{\frac{1}{1+2 l}}]}}\!\frac{1}{i} \mathbb{E}\left\| x_{i}\right\|\left|y_{i+1}\right|^{l}\nonumber\\ 
&+c\sum_{i=n_{k}}^{n_{k}+[n_{k}^{\frac{1}{1+2 l}}]}{\frac{1}{i^{\frac{1+l}{1+2l}}}} \mathbb{E}\left\|x_{i}\right\|^{l+1}
+c\!\!\!\!\sum_{i=n_{k}}^{n_{k}+[n_{k}^{\frac{1}{1+2 l}}]}\!\!\frac{1}{i} \mathbb{E}\left\| x_{i}\right\|.
\end{align}

By Lemma \ref{Le1}, we have as $k \to \infty$
\begin{align}
&\sum_{i=n_{k}}^{n_{k}+{\color{black}[n_{k}^{\frac{1}{1+2 l}}]}}\frac{1}{i} (\left\| x_{i}\right\|\left|y_{i+1}\right|^{l}-\mathbb{E}\left\| x_{i}\right\|\left|y_{i+1}\right|^{l})=o(1),\\
%\end{align}
%\begin{align}
&\sum_{i=n_{k}}^{n_{k}+{\color{black}[n_{k}^{\frac{1}{1+2 l}}]}}{\frac{1}{i^{\frac{1+l}{1+2l}}}}(\left\| x_{i}\right\|^{l+1}-\mathbb{E}\left\|x_{i}\right\|^{l+1})=o(1),
\end{align}
and
\begin{align}
\sum_{i=n_{k}}^{n_{k}+{\color{black}[n_{k}^{\frac{1}{1+2 l}}]}}\frac{1}{i}(\left\| x_{i}\right\|-\mathbb{E}\left\|x_{i}\right\|)=o(1).
\end{align}

By the time-invariant distribution of $\{x_k\}_{k\geq1}$, we have
\begin{align}
&\sum_{i=n_{k}}^{n_{k}+{\color{black}[n_{k}^{\frac{1}{1+2 l}}]}}
\frac{1}{i^{\frac{1+l}{1+2l}}} \mathbb{E}\left\|x_{i}\right\|^{l+1} \nonumber\\
\le&\mathbb{E}\left\|x_{1}\right\|^{l+1}\int_{n_{k}-1}^{n_{k}+{\color{black}[n_{k}^{\frac{1}{1+2 l}}]}} \frac{1}{x^{\frac{1+l}{1+2l} } } \mathrm{d} x\nonumber\\
%=&{\color{black}\frac{1+2l}{l}}\mathbb{E}\left\|x_{1}\right\|^{l+1}\cdot x^{\frac{l}{1+2l}}\Big|_{n_{k}-1}^{n_{k}+{\color{black}[n_{k}^{\frac{1}{1+2 l}}]}}\nonumber\\
=& {\color{black}\frac{1+2l}{l}}\mathbb{E}\left\|x_{1}\right\|^{l+1}\!\! \left[\left ( n_{k}\!+\!\![n_{k}^{\frac{1}{1+2 l}}] \right ) ^{\frac{l}{1+2l}}\!\!\!\!\!-\!\!
\left ( n_{k}\!-\!1 \right ) ^{\frac{l}{1+2l}}\right] 
\underset{k \rightarrow \infty}{\longrightarrow} 0.
\end{align}

Noticing (\ref{nk_l}), we have as $k \to \infty$
\begin{align}
\sum_{i=n_{k}} ^{n_{k}+[n_{k}^{\frac{1}{1+2 l}}]}\!\! \frac{1}{i} \mathbb{E} \left\| x_{i}\right\|\left|y_{i+1}\right|^{l}=O\Big(\sum_{i=n_{k}} ^{n_{k}+{\color{black}[n_{k}^{\frac{1}{1+2 l}}]}} \frac{1}{i}\Big)=o(1),
\end{align}	
and
\begin{align}\label{traninf1}
\sum_{i=n_{k}} ^{n_{k}+{\color{black}[n_{k}^{\frac{1}{1+2 l}}]}} \frac{1}{i} \mathbb{E} \left\| x_{i}\right\|=o(1).
\end{align}
	
It follows from (\ref{truncinf})--(\ref{traninf1}) that as $k\to\infty$,
\begin{align}\label{upbd}
&\big\|\theta_{n_{k}}+\sum_{i=n_{k}}^{n_{k}+[n_{k}^{\frac{1}{1+2 l}}]} a_{i}x_{i} \varphi\left(y_{i+1}-\theta_{i}^{T}  x_{i}\right) \big\|\nonumber\\
\leqslant &\left\|\theta_{n_{k}}\right\|+o(1) <\frac{5}{4}\alpha.
\end{align}

Thus for all $k$ sufficiently large, there is no truncation for $\theta_{i+1}, i=n_{k}, \ldots, n_{k}+{\color{black}[n_{k}^{\frac{1}{1+2 l}}]}$ and
\begin{align}\label{de}
\left\|\theta_{i+1}\right\|<\frac{5}{4} \alpha, ~\left\|\theta_{i+1}-\theta_{n_{k}}\right\| \le C T,
\end{align}
for some $C>0$.

Assume that (\ref{de}) holds for $\theta_{i+1}, i=n_{k}, \ldots, n_{k}+{\color{black}[n_{k}^{\frac{1}{1+2 l}}]}+j<m\left(n_{k}, T\right)$. Next we prove (\ref{de}) also holds for $\theta _{n_{k}+{\color{black}[n_{k}^{\frac{1}{1+2 l}}]}+j+2}$.

By (\ref{de}), we have
\begin{align}
&\Big\| \theta_{ n_{k}+[n_{k}^{\frac{1}{1+2 l}}]+1}+\sum_{i=n_{k}+[n_{k}^{\frac{1}{1+2 l}}]+1}^{n_{k}+[n_{k}^{\frac{1}{1+2 l}}]+j+1} a_{i} x_{i} \varphi(y_{i+1}-\theta_{i}^{T}x_{i})\Big\| \nonumber\\
\le& \big\|\theta_{ n_{k}+[n_{k}^{\frac{1}{1+2 l}}]+1}\big\| +
c\sum_{i=n_{k}+[n_{k}^{\frac{1}{1+2 l}}]+1}^{n_{k}+[n_{k}^{\frac{1}{1+2 l}}]+j+1} a_{i}\left\|x_{i}\right\|\big(| y_{i+1}|^{l}\nonumber\\
&+\left\| \theta_{i}\right\|^{l}\left\|x_{i}\right\|^{l}+1\big) \nonumber\\
\le& \big\|\theta_{ n_{k}+[n_{k}^{\frac{1}{1+2 l}}]+1}\big\| +
c\sum_{i=n_{k}+[n_{k}^{\frac{1}{1+2 l}}]+1}^{n_{k}+[n_{k}^{\frac{1}{1+2 l}}]+j+1} \frac{1}{i}\left\|x_{i}\right\|\left| y_{i+1}\right|^{l} \nonumber\\
&+c\!\!\!\!\!\sum_{i=n_{k}+[n_{k}^{\frac{1}{1+2 l}}]+1}^{n_{k}+[n_{k}^{\frac{1}{1+2 l}}]+j+1}\frac{1}{i}\|x_{i}\|^{l+1}\cdot\!(\frac{5}{4}\alpha)^l\!
\!+\!c\!\!\!\!\sum_{i=n_{k}+[n_{k}^{\frac{1}{1+2 l}}]+1}^{n_{k}+[n_{k}^{\frac{1}{1+2 l}}]+j+1} \frac{1}{i} \left\|x_{i}\right\| \nonumber\\
=&\big\|\theta_{ n_{k}+[n_{k}^{\frac{1}{1+2 l}}]+1}\big\| +
c\sum_{i=n_{k}+[n_{k}^{\frac{1}{1+2 l}}]+1}^{n_{k}+[n_{k}^{\frac{1}{1+2 l}}]+j+1} \frac{1}{i}\Big(\left\|x_{i}\right\|\left| y_{i+1}\right|^{l}\nonumber\\
&-\mathbb{E}\left\|x_{i}\right\|\left| y_{i+1}\right|^{l} +\mathbb{E}\left\|x_{i}\right\|\left| y_{i+1}\right|^{l}\Big)\nonumber\\
&+\!\!(\frac{5}{4}\alpha)^l\cdot c\!\!\!\!\sum_{i=n_{k}+[n_{k}^{\frac{1}{1+2 l}}]+1}^{n_{k}+[n_{k}^{\frac{1}{1+2 l}}]+j+1}\!\!\!\!\frac{1}{i}\Big(\left\|x_{i}\right\|^{l+1}\!-\!\mathbb{E}\left\|x_{i}\right\|^{l+1}+\!\mathbb{E}\left\|x_{i}\right\|^{l+1}\!\Big)\!\nonumber\\
&+\!c\!\!\!\sum_{i=n_{k}+[n_{k}^{\frac{1}{1+2 l}}]+1}^{n_{k}+[n_{k}^{\frac{1}{1+2 l}}]+j+1} \frac{1}{i} \left(\left\|x_{i}\right\|\!-\!\mathbb{E}\left\|x_{i}\right\|\!+\!\mathbb{E}\left\|x_{i}\right\|\right).\label{85'}
\end{align}

By {\color{black}(\ref{xyl}), (\ref{xl+1}), and (\ref{x1})} and the assumption $n_{k}+[n_{k}^{\frac{1}{1+2 l}}]+j+1<m\left(n_{k}, T\right)$, it yields that as $k\to\infty$,
\begin{align}
\sum_{i=n_{k}+[n_{k}^{\frac{1}{1+2 l}}]+1}^{n_{k}+[n_{k}^{\frac{1}{1+2 l}}]+j+1} &\frac{1}{i}(\left\|x_{i}\right\|\left| y_{i+1}\right|^{l}-\mathbb{E}\left\|x_{i}\right\|\left| y_{i+1}\right|^{l}\nonumber\\
&+\mathbb{E}\left\|x_{i}\right\|\left| y_{i+1}\right|^{l})=o(1)+O(T),\label{86'}\\
\sum_{i=n_{k}+[n_{k}^{\frac{1}{1+2 l}}]+1}^{n_{k}+[n_{k}^{\frac{1}{1+2 l}}]+j+1}&\frac{1}{i}(\left\|x_{i}\right\|^{l+1} - \mathbb{E}\left\|x_{i}\right\|^{l+1}\nonumber\\
&+\mathbb{E}\left\|x_{i}\right\|^{l+1})=o(1)+O(T),\label{87'}
\end{align}
and
\begin{align}
\sum_{i=n_{k}+[n_{k}^{\frac{1}{1+2 l}}]+1}^{n_{k}+[n_{k}^{\frac{1}{1+2 l}}]+j+1} \frac{1}{i} \left(\left\|x_{i}\right\|-\mathbb{E}\left\|x_{i}\right\|+\mathbb{E}\left\|x_{i}\right\|\right)=o(1)+O(T).\label{88'}
\end{align}

Noting (\ref{de})--(\ref{88'}), we have
\begin{align}
&\Big\| \theta_{ n_{k}+[n_{k}^{\frac{1}{1+2 l}}]+1}+\sum_{i=n_{k}+[n_{k}^{\frac{1}{1+2 l}}]+1}^{n_{k}+[n_{k}^{\frac{1}{1+2 l}}]+j+1} a_{i} x_{i} \varphi\left(y_{i+1}-\theta_{i}^{T}x_{i}\right)\Big\|\nonumber\\
\le& \frac{5}{4} \alpha+o(1)+O(T)<2 \alpha,
\end{align}
for all $k$ sufficiently large and $T $ sufficiently small.

Hence, there is no truncation for  $\theta_{n_{k}+[n_{k}^{\frac{1}{1+2 l}}]+j+2}$ and
\begin{align}
\big\|\theta_{n_{k}+[n_{k}^{\frac{1}{1+2 l}}]+j+2}-\theta_{n_{k}}\big\| \le C T
\end{align}
for some $C>0$.

By induction we obtain that (\ref{de}) holds for $i=n_{k}, \ldots, m\left(n_{k}, T\right)$. This finishes the proof.
	
\section{Proof of Lemma \ref{Lem4}}

Denote by $\mathcal{S}^{d}$ a countable dense set in $\mathbb{R}^{d}$. By Lemma \ref{Le1}, there exists a set $\Omega_{0}\subset \Omega$ with  $\mathbb{P}\left\{\Omega_{0}\right\}=1$ such that on the trajectory of any $\omega \in \Omega_{0}$, (\ref{e3})--(\ref{e6}) hold and (\ref{e4}) takes place for all $\theta \in \mathcal{S}$. In the following analysis, we focus on a fixed $\omega\in \Omega_{0}$.

Denote by $\bar{\theta}$ the limit of $\left\{\theta_{n_{k}}\right\}$, i.e., $\theta_{n_{k}} \underset{k \to \infty}{\longrightarrow} \bar{\theta}$  and  {\color{black}$\|\theta_{n_{k}}\|<\alpha, k \geq 1 $} for some $\alpha>0$. Let  $\{\theta(n)\}_{n \geq 1}\subset \mathcal{S}^{d}$ be a sequence satisfying $\theta(n) \underset{n \to \infty}{\longrightarrow} \bar{\theta}$. Without loss of generality we assume $\left\|\theta(n)\right\|<\alpha, n\ge1$.

The noise $\varepsilon_{i+1},i=n_{k},\ldots,m\left(n_{k},T\right)$ can be divided into the following parts:
\begin{align}
\varepsilon_{i+1}&=x_{i}\varphi\!\left(y_{i+1}\!-\!\theta_{i}^{T}x_{i} \right) \!-\!\mathbb{E}\!\left[x_{i} \varphi\left(y_{i+1}\!-\!\theta^{T} x_{i}\right)\!\right]\!\!\Big|_{\theta =\theta _{i}}\\
&=\varepsilon_{i+1}^{(1)}(n)+\varepsilon_{i+1}^{(2)}(n)+\varepsilon_{i+1}^{(3)}(n)+\varepsilon_{i+1}^{(4)},
\end{align}	
where
\begin{align}
&\varepsilon_{i+1}^{(1)}(n)  \!\triangleq\! x_{i} \varphi\left(y_{i+1}\!-\!\theta_{i}^{T} x_{i}\right)\!-\!x_{i} \varphi\left(y_{i+1}\!-\!\theta^{T}(n) x_{i}\right),\\
&\varepsilon_{i+1}^{(2)}(n)  
\triangleq \nonumber\\&x_{i} \varphi\left(y_{i+1}\!-\!\theta^{T}(n) x_{i}\right)\!-\!\mathbb{E}\!\left[x_{i} \varphi\left(y_{i+1}\!-\!\theta^{T} x_{i}\right)\right]\Big|_{\theta = \theta(n)},\\
&\varepsilon_{i+1}^{(3)}(n)  
\triangleq \nonumber\\
&\mathbb{E}\!\left[x_{i} \varphi\left(y_{i+1}\!-\!\theta^{T} x_{i}\right)\!\right]\!\!\Big|_{\theta=\theta(n)}\!\!\!\!\!\!\!\!-\mathbb{E}\!\left[x_{i} \varphi\left(y_{i+1}\!-\! \theta^{T} x_{i}\right)\!\right]\!\!\Big|_{\theta=\bar{\theta}},\\
&\varepsilon_{i+1}^{(4)}\triangleq\nonumber\\&
\mathbb{E}\!\left[x_{i} \varphi\left(y_{i+1}\!-\!\theta^{T} x_{i}\right)\!\right]\!\!\Big|_{\theta=\bar{\theta}}\!\!\!\!-\mathbb{E}\!\left[x_{i} \varphi\left(y_{i+1}\!-\!\theta^{T} x_{i}\right)\!\right]\!\!\Big|_{\theta = \theta_{i}}.
\end{align}

To prove the lemma it suffices to verify (\ref{noise}) with $\varepsilon_{i+1}$ replaced by $\varepsilon_{i+1}^{(j)}(n), j=1, \cdots, 4$. We begin with $\varepsilon_{i+1}^{(1)}(n)$.

For $\varepsilon_{i+1}^{(1)}(n)$, we have
\begin{align}
&\frac{1}{T}\Big\|\sum_{i=n_{k}}^{m\left(n_{k}, T\right)} a_{i} \varepsilon_{i+1}^{(1)}(n)\Big\|  \nonumber \\
%=&\frac{1}{T}\Big\|\sum_{i=n_{k}}^{m\left(n_{k}, T\right)} a_{i}x_{i} \left[\varphi(y_{i+1}-\theta_{i}^{T}x_{i})-\varphi(y_{i+1}-\theta^{T}(n)x_{i})\right]\Big\|\nonumber\\
%\le&  \frac{1}{T}\sum_{i=n_{k}}^{m\left(n_{k}, T\right)} a_{i}\left\|x_{i}\right\| \left| \varphi(y_{i+1}-\theta_{i}^{T}x_{i})-\varphi(y_{i+1}-\theta^{T}(n)x_{i})\right|\nonumber\\
=&\frac{1}{T}\sum_{i=n_{k}}^{m\left(n_{k}, T\right)} a_{i}\left\|x_{i}\right\| \left| \varphi(y_{i+1}-\theta_{i}^{T}x_{i})-\varphi(y_{i+1}-\theta^{T}(n)x_{i})\right|\nonumber\\
&\cdot\left(\mathbb{I}_{\left[\left|\theta_{i}^{T}x_{i}-\theta^{T}(n)x_{i}\right|>1 \right]}+\mathbb{I}_{\left[\left|\theta_{i}^{T}x_{i}-\theta^{T}(n)x_{i}\right| \leq 1\right]}\right).
\end{align}

Note that {\color{black}$\|\theta_{n_{k}}\|<\alpha,~\forall k \ge 1$, } $\left\|\bar{\theta}(n)\right\|<\alpha,~\forall n\ge1$ and by Lemma \ref{Le2}  {\color{black}$\|\theta_{i+1}\|<2 \alpha,~i=n_{k}, \ldots, m\left(n_{k}, T\right)$} for all $k$ sufficiently large. We have the following chain of inequalities:
\begin{align}\label{noise1}
\frac{1}{T}\!\!\!\sum_{i=n_{k}}^{m\left(n_{k}, T\right)}\!\!\! &a_{i}\left\|x_{i}\right\| \left| \varphi\left(y_{i+1}\!-\!\theta_{i}^{T}x_{i}\right)-\varphi\left(y_{i+1}\!-\!\theta^{T}(n)x_{i}\right)\right|\nonumber\\
&\cdot\mathbb{I}_{\left[\left|\theta_{i}^{T}x_{i}-\theta^{T}(n)x_{i}\right|>1 \right]}\nonumber\\
%\le \frac{1}{T}\!\!\!\sum_{i=n_{k}}^{m\left(n_{k}, T\right)} %&\!\!\!a_{i}\left\|x_{i}\right\| \left(\left| \varphi\left(y_{i+1}\!-\!\theta_{i}^{T}x_{i}\right)\right|\!\!+\!\!\left|\varphi\left(y_{i+1}\!-\!\theta^{T}(n)x_{i}\right)\right|\right)\nonumber\\
%&\cdot\left|\theta_{i}^{T}x_{i}-\theta^{T}(n)x_{i}\right| \nonumber\\
\le  \frac{c}{T}\sum_{i=n_{k}}^{m\left(n_{k}, T\right)}&\frac{1}{i}\left\|x_{i}\right\|^{2}\left(\left|y_{i+1}\right|^{l}+\left\|x_{i}\right\|^{l}+1\right)\left\|\theta_{i}-\theta(n)\right\| \nonumber \\
 \le  \frac{c}{T}\sum_{i=n_{k}}^{m\left(n_{k}, T\right)}&\frac{1}{i}(\left\|x_{i}\right\|^{2}\left|y_{i+1}\right|^{l}+\left\|x_{i}\right\|^{l+2} +{\color{black}\|x_i\|^2})\nonumber\\
 &\cdot(\left\|\theta_{i}-\theta_{n_{k}}\right\|+\left\|\theta_{n_{k}}-\bar{\theta}\right\|+\left\|\theta(n)-\bar{\theta}\right\|) \nonumber \\
\le  \frac{c}{T}\sum_{i=n_{k}}^{m\left(n_{k}, T\right)}&\frac{1}{i}(\left\|x_{i}\right\|^{2}\left|y_{i+1}\right|^{l}+\left\|x_{i}\right\|^{l+2} +{\color{black}\|x_i\|^2})\nonumber\\
&\cdot(cT+\left\|\theta_{n_{k}}-\bar{\theta}\right\|+\left\|\theta(n)-\bar{\theta}\right\|),
\end{align}
where $c$ is a positive constant which may vary among different lines of inequalities.

By Lemma \ref{Le1} and the definition of $m\left(n_{k}, T\right)$, we have
\begin{align}\label{x^appr}
&\sum_{i=n_{k}}^{m\left(n_{k}, T\right)}\frac{1}{i}(\left\|x_{i}\right\|^{2}\left|y_{i+1}\right|^{l}+\left\|x_{i}\right\|^{l+2}+{\color{black}\left\|x_{i}\right\|^{2}})\nonumber\\
=&\sum_{i=n_{k}}^{m\left(n_{k}, T\right)}\frac{1}{i}\Big[\left\|x_{i}\right\|^{2}\left|y_{i+1}\right|^{l}-\mathbb{E}\left\|x_{i}\right\|^{2}\left|y_{i+1}\right|^{l} +\left\|x_{i}\right\|^{l+2}\nonumber\\
&~~ -\mathbb{E}\left\|x_{i}\right\|^{l+2} 
+\left\|x_{i}\right\|^{2}-\mathbb{E}\left\|x_{i}\right\|^{2}\nonumber\\
&~~+\mathbb{E}\left\|x_{i}\right\|^{2}\left|y_{i+1}\right|^{l}+\mathbb{E}\left\|x_{i}\right\|^{l+2} +{\color{black}\mathbb{E}\left\|x_{i}\right\|^{2}}\Big]\nonumber\\
=&o(1)+O(T),
\end{align}
as $k\to \infty$ and hence
\begin{align}\label{fenduan}
\lim _{n \rightarrow \infty} &\lim _{T \rightarrow 0} \limsup _{k \rightarrow \infty} \frac{1}{T}\sum_{i=n_{k}}^{m\left(n_{k}, T\right)} a_{i}\left\|x_{i}\right\| | \varphi(y_{i+1}-\theta_{i}^{T}x_{i})\nonumber\\
&-\!\varphi(y_{i+1}\!-\!\theta^{T}(n)x_{i})|\cdot\mathbb{I}_{\left[\left|\theta_{i}^{T}x_{i}\!-\!\theta^{T}(n)x_{i}\right|>1 \right]}=0.
\end{align}

Now we consider the case $|\theta_{i}^{T}x_{i}-\theta^{T}(n)x_{i}|\le1$. We first consider the case if {\color{black}$\left|y_{i+1}-\theta_{i}^{T} x_{i}\right|>a+1$, where $a$ is specified in the assumption A1'),} which leads to {\color{black}$\left|y_{i+1}-\theta^{T}(n) x_{i}\right|>a$.} By noting that by the assumption A1'), on the interval $(-\infty, 0) \cup(0, \infty)$ the function $\varphi(\cdot)$ has a continuous derivative $\varphi^{(1)}(\cdot)$ and by the mean value theorem, there exists $r_{i}\in(0,1)$ such that
\begin{align}\label{meanThe}
\frac{1}{T} &\sum_{i=n_{k}}^{m\left(n_{k}, T\right)} \frac{1}{i} \left\|x_{i}\right\|\left|\varphi(y_{i+1}-\theta_{i}^{T} x_{i})-\varphi(y_{i+1}-\theta^{T}(n) x_{i})\right|\nonumber\\
&\cdot \mathbb{I}_{[|\theta_{i}^{T}x_{i}-\theta^{T}(n)x_{i}|\le1,|y_{i+1}-\theta_{i}^{T} x_{i}|>{\color{black}a+1}]} \nonumber\\
\le \frac{1}{T}\!\! &\sum_{i=n_{k}}^{m\left(n_{k}, T\right)} \!\!\!\frac{1}{i}\left\|x_{i}\right\| \left|\varphi^{(1)}(y_{i+1}-r_{i} \theta_{i}^{T} x_{i}-(1-r_{i}) \theta^{T}(n) x_{i})\right|\nonumber\\
&\cdot\left|\theta_{i}^{T} x_{i}-\theta^{T}(n) x_{i}\right|\cdot\!\! \mathbb{I}_{[|\theta_{i}^{T}x_{i}-\theta^{T}(n)x_{i}|\le1,|y_{i+1}-\theta_{i}^{T} x_{i}|>a+1]} \nonumber\\
\le \frac{c}{T} &\sum_{i=n_{k}}^{m\left(n_{k}, T\right)} \frac{1}{i}\left\|x_{i}\right\|^2(|y_{i+1}|^l+ \left\|x_{i}\right\|^l+1)\nonumber\\
&\cdot\left(\left\|\theta_{i}-\theta_{n_{k}}\right\|+\left\|\theta_{n_{k}}-\bar{\theta}\right\|+ \left\|\theta(n)-\bar{\theta}\right\|\right),
\end{align}
where $c>0$ is a positive constant.

Combining {\color{black}(\ref{cT}), }(\ref{x^appr}) and (\ref{meanThe}), we have
\begin{align}
&\lim _{n \rightarrow \infty} \lim _{T \rightarrow 0} \limsup _{k \rightarrow \infty}\frac{1}{T} \sum_{i=n_{k}}^{m\left(n_{k}, T\right)}\frac{1}{i} \left\|x_{i}\right\|\big|\varphi(y_{i+1}-\theta_{i}^{T} x_{i})\nonumber\\
&\!-\!\varphi(y_{i+1}\!-\!\theta^{T}\!(n) x_{i})\big|\!
\cdot\!(\mathbb{I}_{[|\theta_{i}^{T}x_{i}\!-\!\theta^{T}(n)x_{i}|\le1,|y_{i+1}\!-\!\theta_{i}^{T} x_{i}|>{\color{black}a+\!1}]})\!=\!0.\label{noise1_2}
\end{align}

While if {\color{black}$\left|y_{i+1}-\theta_{i}^{T} x_{i}\right|\le a+1$,} we have {\color{black}$\left|y_{i+1}-\theta^{T}(n) x_{i}\right| \le a+2$.}
Since $\varphi(t)$ is continuous on $\mathbb{R}$ and thus is uniformly continuous on {\color{black}$[-a-2, a+2]$,} we have that for any $\varepsilon>0$, there exists $\delta>0$ such that
\begin{equation}\label{daoshuunic}
\left|\varphi\left(t_1\right)\!-\!\varphi\left(t_2\right)\right|<\varepsilon,~~\forall\left|t_1-t_2\right|<\delta,~~t_1, t_2 \in{\color{black}[-a\!-\!2, a\!+\!2]},
\end{equation}

For the fixed $\varepsilon>0$, if $|\theta_{i}^{T}x_{i}-\theta^{T}(n)x_{i}|<\delta$, then by (\ref{daoshuunic}) we have
\begin{align}\label{unicont}
&\frac{1}{T} \sum_{i=n_{k}}^{m\left(n_{k}, T\right)} \frac{1}{i} \left\|x_{i}\right\||\varphi(y_{i+1}-\theta_{i}^{T} x_{i})-\varphi(y_{i+1}-\theta^{T}(n) x_{i})|\nonumber\\
&\cdot \mathbb{I}_{[|\theta_{i}^{T}x_{i}-\theta^{T}(n)x_{i}|\le1,|y_{i+1}-\theta_{i}^{T} x_{i}|\le {\color{black}a+1},|\theta_{i}^{T}x_{i}-\theta^{T}(n)x_{i}|<\delta]} \nonumber \\
\le& \frac{\varepsilon}{T}  \sum_{i=n_{k}}^{m\left(n_{k}, T\right)} \frac{1}{i} \left\|x_{i}\right\| = \frac{\varepsilon}{T}(o(1)+O(T)).
\end{align}

%It means that if $k$ is sufficiently large, $T$ is sufficiently small, then
%\begin{align}\label{delta<}
%	\frac{1}{T} \sum_{i=n_{k}}^{m\left(n_{k}, T\right)} \frac{1}{i} \left\|x_{i}\right\||\varphi(y_{i+1}-\theta_{i}^{T} x_{i})-\varphi(y_{i+1}-\theta^{T}(n) x_{i})|(\mathbb{I}_{[|\theta_{i}^{T}x_{i}-\theta^{T}(n)x_{i}|\le1,|y_{i+1}-\theta_{i}^{T} x_{i}|\le 3,|\theta_{i}^{T}x_{i}-\theta^{T}(n)x_{i}|<\delta]})<c\varepsilon.
%\end{align}

For the case that $\left|y_{i+1}-{{\color{black}\theta_i}}^{T} x_{i}\right|\le {\color{black}a+1}$ and $|\theta_{i}^{T}x_{i}-\theta^{T}(n)x_{i}|\ge \delta$, it yields
$$
1\le\frac{\left|\theta_{i}^{T} x_{i}-\theta^{T}(n) x_{i}\right|}{\delta}.
$$

Similarly to the derivation of (\ref{noise1}) and by (\ref{x^appr}), we have
\begin{align}
&\frac{1}{T} \sum_{i=n_{k}}^{m\left(n_{k}, T\right)} \frac{1}{i} \left\|x_{i}\right\||\varphi(y_{i+1}-\theta_{i}^{T} x_{i})-\varphi(y_{i+1}-\theta^{T}(n) x_{i})|\nonumber\\
&\cdot\mathbb{I}_{[|\theta_{i}^{T}x_{i}-\theta^{T}(n)x_{i}|{\color{black}\le}1,|y_{i+1}-\theta_{i}^{T} x_{i}|\le {\color{black}a+1},|\theta_{i}^{T}x_{i}-\theta^{T}(n)x_{i}|\ge \delta]} \nonumber\\
\le&	\frac{1}{T} \sum_{i=n_{k}}^{m\left(n_{k}, T\right)} \frac{1}{i} \left\|x_{i}\right\||\varphi(y_{i+1}-\theta_{i}^{T} x_{i})-\varphi(y_{i+1}-\theta^{T}(n) x_{i})|\nonumber\\
&~~~~~~~~~~~\cdot\frac{\left|\theta_{i}^{T} x_{i}-\theta^{T}(n) x_{i}\right|}{\delta} \nonumber\\
\le & \frac{1}{T\delta}\!\!\!\!\! \sum_{i=n_{k}}^{m\left(n_{k}, T\right)} \!\!\frac{1}{i} \left\|x_{i}\right\|(|\varphi(y_{i+1}-\theta_{i}^{T} x_{i})|\!+\!|\varphi(y_{i+1}-\theta^{T}(n) x_{i})|)\nonumber\\
&~~~~~~~~~~~\cdot|\theta_{i}^{T} x_{i}-\theta^{T}(n) x_{i}| \nonumber\\
\le & \frac{c}{T\delta}\sum_{i=n_{k}}^{m\left(n_{k}, T\right)}\frac{1}{i}(\left\|x_{i}\right\|^{2}\left|y_{i+1}\right|^{l}+\left\|x_{i}\right\|^{l+2} +{\color{black}{\left\|x_{i}\right\|^{2}}})\nonumber\\
&(\left\|\theta_{i}-\theta_{n_{k}}\right\| +\left\|\theta_{n_{k}}-\bar{\theta}\right\|+\left\|\theta(n)-\bar{\theta}\right\|) \nonumber \\
\label{epsilon}\le & \frac{c}{T \delta}(o(1)+O(T)) \cdot({\color{black}O(T)}+o(1)+\|\theta(n)-\bar{\theta}\|),
\end{align}
where $c>0$ is a constant.

%Also, we can choose $k$ sufficiently large, $T$ sufficiently small, $n$ sufficiently large such that
%\begin{align}\label{delta>}
% 	\frac{1}{T} \sum_{i=n_{k}}^{m\left(n_{k}, T\right)} \frac{1}{i} \left\|x_{i}\right\||\varphi(y_{i+1}-\theta_{i}^{T} x_{i})-\varphi(y_{i+1}-\theta^{T}(n) x_{i})|(\mathbb{I}_{[|\theta_{i}^{T}x_{i}-\theta^{T}(n)x_{i}|\le1,|y_{i+1}-\theta_{1}^{T} x_{i}|\le 3, |\theta_{i}^{T}x_{i}-\theta^{T}(n)x_{i}|\ge \delta]})<\varepsilon.
%\end{align}

Combining (\ref{fenduan}), (\ref{noise1_2}), (\ref{unicont}), and (\ref{epsilon}), we have

\begin{align}
\lim _{\varepsilon \rightarrow 0}\lim _{n \rightarrow \infty} \lim _{T \rightarrow 0} \limsup _{k \rightarrow \infty} \frac{1}{T}\Big\|\sum_{i=n_{k}}^{m\left(n_{k}, T\right)} a_{i} \varepsilon_{i+1}^{(1)}(n)\Big\|=0.\label{107'}
\end{align}

For $\varepsilon_{i+1}^{(2)}(n)$, by Lemma {\color{black}\ref{Le1}} we know that
\begin{align}
\sum_{k=1}^{\infty} &\frac{1}{k}\Big( x_{k}\varphi(y_{k+1}-\theta^{T}(n)x_{k})\nonumber\\
&-\big[\mathbb{E}x_{k}\varphi(y_{k+1}-\theta^{T}x_{k})\big]\Big|_{\theta = \theta(n)}\Big)<\infty ~~ \forall n\ge1,
\end{align}
from which it yields
\begin{align}
\lim_{n \rightarrow \infty} \lim _{T \rightarrow 0}
\limsup_{k \rightarrow \infty} \frac{1}{T}\Big\|\sum_{i=n_{k}}^{m\left(n_{k}, T\right)} a_{i} \varepsilon_{i+1}^{(2)}(n)\Big\|=0.\label{109'}
\end{align}

For $\varepsilon_{i+1}^{(3)}(n)$, since $\{w_{k}\}_{k\geq1}$ is i.i.d. and $\{x_k\}_{k\ge1}$ has a time-invariant pdf, we have
\begin{align}
&\frac{1}{T} \Big\|\sum_{i=n_{k}}^{m\left(n_{k}, T\right)} \frac{1}{i} \mathbb{E}\left[x_{i} \varphi\left(y_{i+1}-z_{1}^{T} x_{i}\right)\right]\Big|_{z_{1}=\theta(n)}\nonumber\\
&~~-\mathbb{E}\left[x_{i} \varphi\left(y_{i+1}-z_{2}^{T} x_{i}\right)\right]\Big|_{z_{2}=\bar{\theta}} \Big\| \nonumber\\
\le & \frac{1}{T} \sum_{i=n_{k}}^{m\left(n_{k}, T\right)} \frac{1}{i} \mathbb{E} \Big[\left\|x_{i}\right\|\nonumber\\ &\cdot\left|\varphi\left(y_{i+1}-z_{1}^{T} x_{i}\right)-\varphi\left(y_{i+1}-z_{2}^{T} x_{i}\right)\right|\Big|_{z_{1}=\theta(n),z_{2}=\bar{\theta}}\Big]\nonumber\\
\le& \mathbb{E}\! \left\|x_{1}\right\| \!\left|\varphi\left(y_{2}\!-\!z_{1}^{T} x_{1}\right)\!-\!\varphi\left(y_{2}\!-\!z_{2}^{T} x_{1}\right)\right|\!\!\Big|_{z_{1}=\theta(n),z_{2}=\bar{\theta}}.\label{110'}
\end{align}
By the property of $\varphi(\cdot)$ and noting (\ref{81}), then by the Lebesgue dominated convergence theorem it holds that
\begin{align}\label{noise3}
\lim _{n \rightarrow \infty} \lim _{T \rightarrow 0} \limsup _{k \rightarrow \infty} \frac{1}{T}\Big\|\sum_{i=n_{k}}^{m\left(n_{k}, T\right)} a_{i} \varepsilon_{i+1}^{(3)}(n)\Big\|=0.
\end{align}

Carrying out a similar discussion as (\ref{107'}), we can prove that
\begin{align}
\lim _{T \rightarrow 0} \limsup _{k \rightarrow \infty} \frac{1}{T}\Big\|\sum_{i=n_{k}}^{m\left(n_{k}, T\right)} a_{i} \varepsilon_{i+1}^{(4)}\Big\|=0.\label{119'}
\end{align}

Noting (\ref{107'}), (\ref{109'}), (\ref{noise3}), and (\ref{119'}), we can prove (\ref{noise}). This finishes the proof.

\section*{References}

\def\refname{\vadjust{\vspace*{-2.5em}}} %Please don't do this in a real paper.
\bibliographystyle{IEEEtran}
\bibliography{ref}

\begin{IEEEbiography}[{\includegraphics[width=1in,height=1.25in,clip,keepaspectratio]{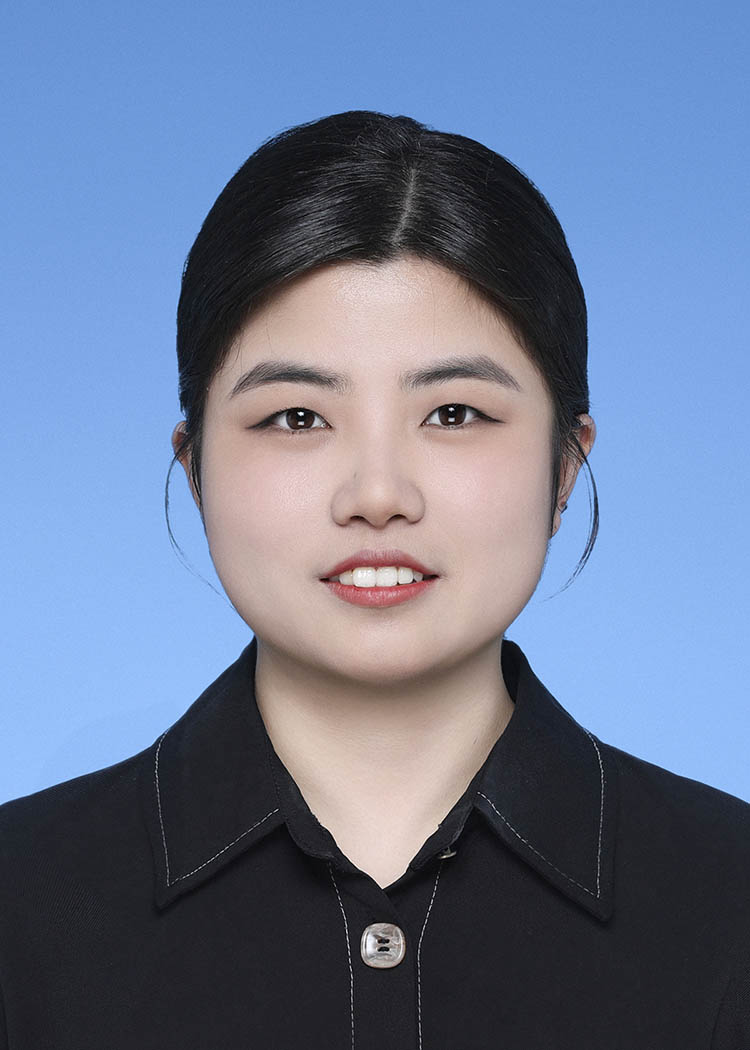}}]
{Mingxia Ding} received her B.S. degree from Dalian University of Technology, China, in 2020. She is currently pursuing her PHD degree at the Institute of Systems Science (ISS), Academy of Mathematics and Systems Science (AMSS), Chinese Academy of Sciences (CAS), China. Her research interests are mainly in system identification, adaptive control and stochastic approximation algorithm.
\end{IEEEbiography}

\begin{IEEEbiography}[{\includegraphics[width=1in,height=1.25in,clip,keepaspectratio]{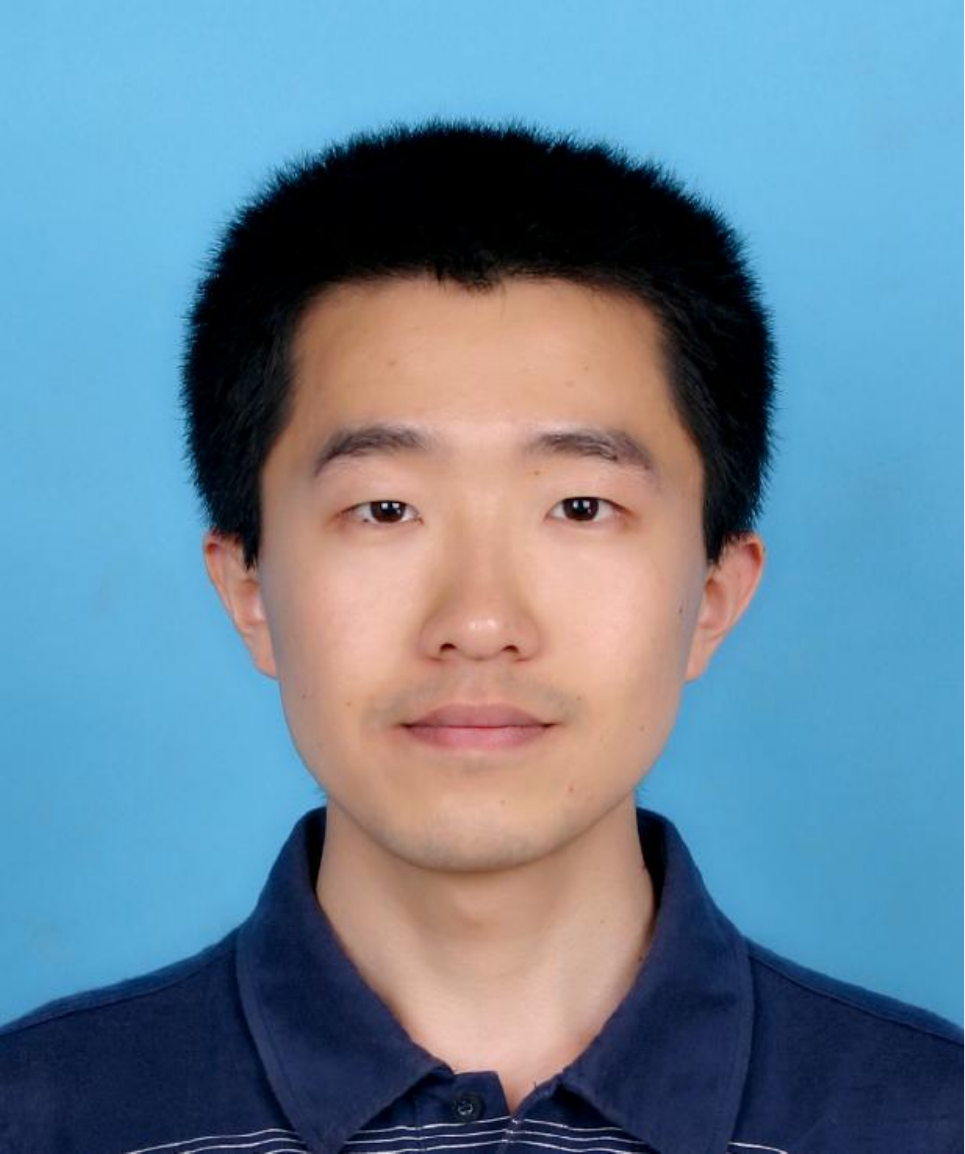}}]
{Wenxiao Zhao} received the Ph.D. degree in operation research and cybernetics from the Institute of Systems Science (ISS), Academy of Mathematics and Systems Science (AMSS), Chinese Academy of Sciences (CAS), China, in 2008. He is currently a Professor with AMSS, CAS. His research interests are mainly in system
identification and adaptive control, variable and feature selection, and distributed stochastic optimization.
\end{IEEEbiography}

\begin{IEEEbiography}[{\includegraphics[width=1in,height=1.25in,clip,keepaspectratio]{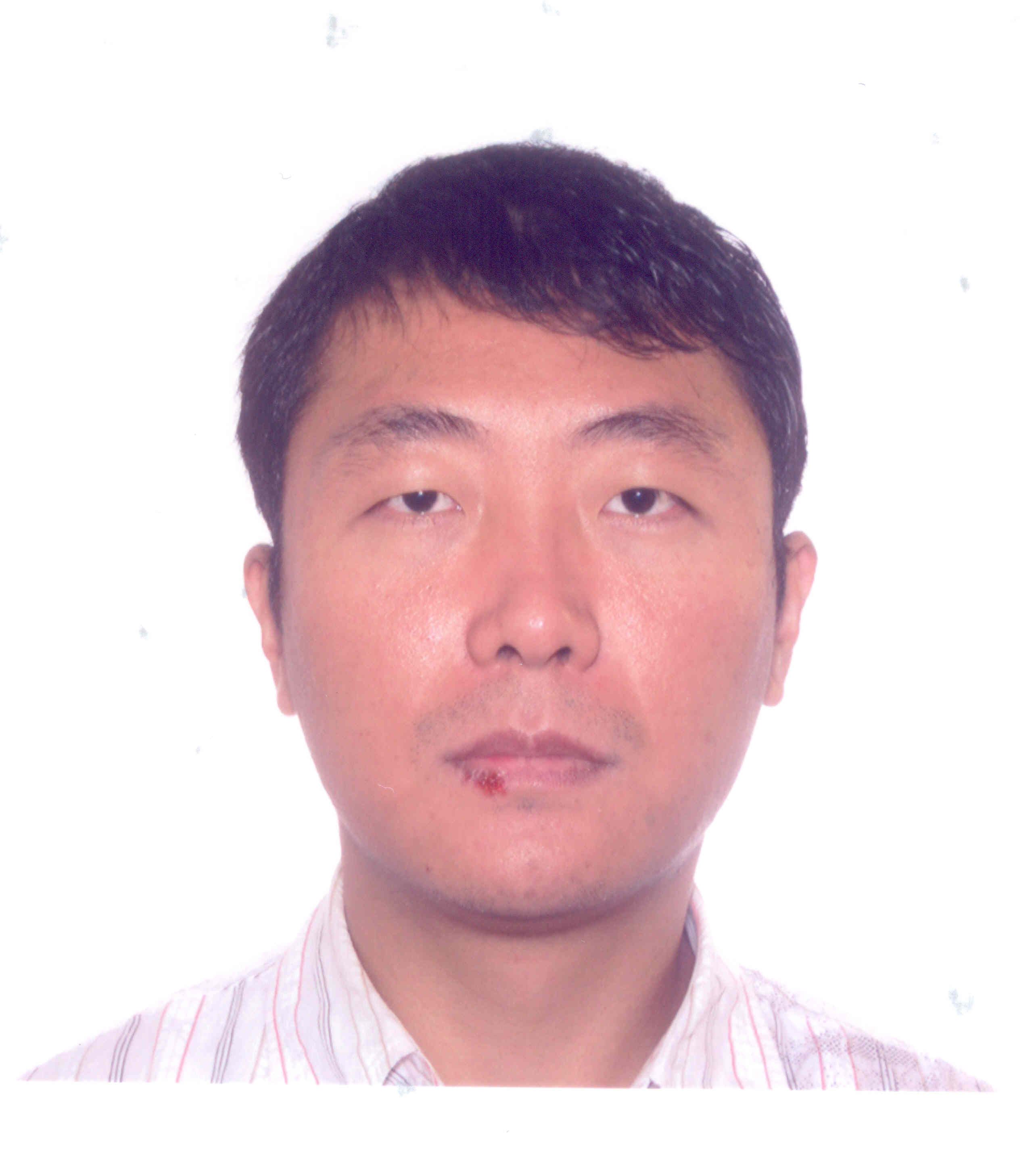}}]
{Tianshi Chen} received the Ph.D. degree in automation and computer-aided engineering from The Chinese University of Hong Kong, Hong Kong, in 2008.  From 2009 to 2015, he was with the Division of Automatic Control, Department of Electrical Engineering, Link\''{o}ping University, Link\''{o}ping, Sweden, first as a Postdoc and then as an Assistant Professor. He is currently a Professor with the Chinese University of Hong Kong, Shenzhen (CUHK-SZ). He has been mainly working in the area of systems and control with focus on system identification, automatic control, and their applications.

Dr. Chen received the Youth Talents Award of the Thousand Talents Plan of China in 2015. He was a plenary speaker at the 19th IFAC Symposium on System Identification, Padova, Italy, 2021, and a coauthor of the book “Regularized System Identification - Learning Dynamic Models from Data.” He is/was an Associate Editor for IEEE TRANSACTIONS ON AUTOMATIC CONTROL (2024–2026), Automatica (2017–2026), System \& Control Letters (2017–2020), and IEEE Control System Society Conference Editorial Board (2016–2019). He received several awards, including the Presidential Research Fellow Award and Presidential Exemplary Teaching Award of CUHK-SZ, in 2020 and 2021, respectively, and the Shenzhen Excellent Teacher Award in 2022.
\end{IEEEbiography}
\end{document}